\documentclass[11pt]{amsart}
\usepackage{latexsym,amssymb,amsmath,youngtab}
\usepackage{youngtab}
\usepackage{epsfig}
\usepackage{amsmath,amsthm,amssymb,amscd,MnSymbol}
\usepackage[mathscr]{eucal}
\usepackage{verbatim, hyperref}
\usepackage{color}
\newcommand*{\textlabel}[2]{%
  \edef\@currentlabel{#1}
  \phantomsection
  #1\label{#2}
}

\newtheoremstyle{custom}
  {3pt}
  {3pt}
  {\slshape}
  {}
  {\bfseries}
  {.}
  { }
   {}
\theoremstyle{custom}
\newtheorem{theorem}{Theorem}[section]
\newtheorem{proposition}[theorem]{Proposition}

\newtheorem{proposition/definition}[theorem]{Proposition/Definition}
\newtheorem{lemma}[theorem]{Lemma}
\newtheorem{corollary}[theorem]{Corollary}
\newtheorem{conjecture}[theorem]{Conjecture}

\theoremstyle{definition}
\newtheorem{definition}[theorem]{Definition}

\newtheorem{example}[theorem]{Example}

\newtheorem{problem}[theorem]{Problem}
\newtheorem{question}[theorem]{Question}

\theoremstyle{remark}
\newtheorem{remark}[theorem]{Remark}

\newcommand{\stack}[2]{\ensuremath{\genfrac{}{}{0pt}{}{#1}{#2}}} 

\newtheoremstyle{exercise}
  {3pt}
  {6pt}
  {}
  {}
  {\bfseries}
  {:}
  { }
   {}
\theoremstyle{exercise}
\newtheorem{exercise}[theorem]{Exercise}
\newtheoremstyle{exercises}
  {3pt}
  {6pt}
  {}
  {}
  {\bfseries}
  {:}
  {\newline}
   {}

\theoremstyle{exercise}
\newtheorem{exercises}[theorem]{Exercises}

\input epsf
\def\boxit#1{\vbox{\hrule height1pt\hbox{\vrule width1pt\kern3pt
  \vbox{\kern3pt#1\kern3pt}\kern3pt\vrule width1pt}\hrule height1pt}}

\def\trank{\text{rank}}

\def\BC{\mathbb C}
\def\BR{\mathbb R}
\def\BP{\mathbb P}
\def\pp#1{\mathbb P^{#1}}

\def\fgl{\mathfrak g\mathfrak l}

\def\pp#1{{\mathbb P}^{#1}}
\def\tdim{{\rm dim}}

\def\hd{,...,}
\def\ww{\wedge}
\def\upperp{{}^\perp}

\def\inv{{}^{-1}}

\def\cA{{\mathcal A}}

\def\cE{{\mathcal E}}

\def\cO{{\mathcal O}}

\def\11{\mathbf 1}

\def\fsl{{\mathfrak {sl}}}

\def\l{\lambda}
\def\a{\alpha}

\def\o{\omega}

\def\b{\beta}
\def\g{\gamma}
\def\s{\sigma}

\def\d{\delta}

\def\ot{{\mathord{ \otimes } }}
\def\op{{\mathord{\,\oplus }\,}}
\def\otc{{\mathord{\otimes\cdots\otimes}\;}}

\def\ra{{\mathord{\;\rightarrow\;}}}

\def\dim{{\rm dim}\;}
\def\La#1{\Lambda^{#1}}

\def\frak{\mathfrak}
\def\fgl{\frak g\frak l}\def\fsl{\frak s\frak l}

\def\op{\oplus}
\def\BZ{\Bbb Z}

\def\ep{\epsilon}
\def\op{\oplus}

\def\s{\sigma}
\def\t{\tau}

\def\a{\alpha}
\def\b{\beta}

\def\g{\gamma}
\def\l{\lambda}

\def\FS{\mathfrak  S}

\def\ol{\overline}

\def\BP{\mathbb  P}
\def\BC{\mathbb  C}

\def\pp#1{\mathbb  P^{#1}}

\def\BR{\mathbb  R}

\def\ep{\epsilon}

\def\hd{, \hdots ,}

\def\inv{{}^{-1}}

\def\La#1{\Lambda^{#1}}

\def\pp#1{\mathbb  P^{#1}}

\def\ur{\underline {\bold R}}

\def\ra{\rightarrow}

\def\tend{\operatorname{End}}

\def\tdim{\operatorname{dim}}

\def\tlim{\lim}

\def\thom{\operatorname{Hom}}
\def\trank{\operatorname{rank}}

\def\upperp{{}^{\perp}}

\def\ww{\wedge}

\def\bbb{{\bold{b}}}

\def\be{\begin{equation}}
\def\ene{\end{equation}}
\def\aaa{{\bold {a}}}
\def\bbb{{\bold {b}}}
\def\ccc{{\bold {c}}}

\DeclareMathOperator{\tlog}{log}
\DeclareMathOperator\tspan{span}

\def\tspan{{\rm span}}

\def\Mn{M_{\langle \nnn \rangle}}\def\Mone{M_{\langle 1\rangle}}

\def\tinf{{\rm inf}}

\def\trank{{\mathrm {rank}}}

\def\aaa{{\bold a}}\def\bbb{{\bold b}}\def\ccc{{\bold c}}

\def\mmm{\bold m}\def\nnn{\bold n}\def\lll{\bold l}

\begin{document}

\title{Abelian Tensors}
\author{J.M. Landsberg}
\address{
Department of Mathematics\\
Texas A\&M University\\
Mailstop 3368\\
College Station, TX 77843-3368, USA}
\email{jml@math.tamu.edu}
\author{Mateusz Micha{\l}ek}
\address{
Freie Universit\"at\\
 Arnimallee 3\\
 14195 Berlin, Germany\newline
Polish Academy of Sciences\\
         ul. \'Sniadeckich 8\\
         00-956 Warsaw\\
         Poland}
\email{wajcha2@poczta.onet.pl}

 \begin{abstract} We analyze tensors in $\BC^m\ot \BC^m\ot \BC^m$  satisfying Strassen's equations for border rank $m$.
 Results include: two purely geometric characterizations of the Coppersmith-Winograd tensor, a reduction to the study of symmetric
 tensors under a mild genericity hypothesis, and numerous additional equations and examples. This study is closely connected to the
 study of the variety of $m$-dimensional abelian subspaces of $\tend(\BC^m)$ and the subvariety consisting of the Zariski closure
 of the variety of maximal tori, called the variety of reductions.
 
 \medskip

\noindent{\bf Sommaire}. Nous \'etudions des tenseurs dans $\BC^m\ot \BC^m\ot \BC^m$ satisfaisant les \'equations de Strassen lorsque le rang du bord   vaut $m$. Les r\'esultats obtenus comprennent : deux caract\'erisations purement g\'eom\'etriques du tenseur de Coppersmith-Winograd, une r\'eduction \`a l'\'etude des tenseurs sym\'etriques sous une hypoth\`ese raisonnable de g\'en\'ericit\'e, et beaucoup de nouveaux exemples et \'equations. Cette \'etude est li\'ee de pr\`es \`a l'\'etude de la vari\'et\'e des sous-espaces ab\'eliens de dimension $m$ de $\tend(\BC^m)$ et la sous-vari\'et\'e obtenue comme  l'adh\'erence de Zariski de la vari\'et\'e des tores maximaux, appel\'ee vari\'et\'e des r\'eductions.
 \end{abstract}
\thanks{Landsberg partially  supported by   NSF grants  DMS-1006353, DMS-1405348. Michalek was supported by Iuventus Plus grant 0301/IP3/2015/73 of the Polish Ministry of Science.}

\maketitle

\section{Introduction}
 The rank and border rank of a tensor $T\in \BC^m\ot \BC^m\ot \BC^m$ (defined below) are basic
measures of its complexity. Central
 problems are to develop techniques to determine them (see, e.g.,
\cite{MR2865915,BCS,Como94:SP,MR3236394}).
Complete resolutions of these problems are currently out of reach. For
example,
neither problem is   solved already in $\BC^4\ot \BC^4\ot \BC^4$.
This article focuses on a very special class of tensors, those satisfying
Strassen's commutativity equations (see \S\ref{streqnssect}). The study of such
tensors is related to the classical
problem of studying spaces of commuting matrices, see, e.g.~\cite{MR0132079,MR2118458,MR1199042,MR813045}.

To completely understand border rank, it would be sufficient
to understand the case of border rank $m$ in $\BC^m\ot \BC^m\ot \BC^m$ (see \cite[Cor. 7.4.1.2]{MR2865915}).
  We study this problem under two
genericity hypotheses - {\it concision},
which essentially says we restrict to tensors that are not contained in some
$\BC^{m-1}\ot \BC^m\ot \BC^m$, and {\it $1_A$-genericity}, which is defined
below.
Even under these genericity hypotheses, the problem is still subtle.

Let $A,B,C$ be complex vector spaces of dimensions $\aaa,\bbb,\ccc$, let
$T\in A\ot B\ot C$ be a tensor. (In bases $T$ is a  three dimensional matrix of
size $\aaa\times \bbb\times \ccc$.) We may view $T$ as a linear map
$T : A^*\ra B\ot C\simeq \thom(C^*,B)$. (In bases, $T ((\a_1\hd \a_{\aaa}))$
is the $\bbb\times \ccc$ matrix
$\a_1$ times the first slice of the $\aaa\times \bbb\times \ccc$ matrix, plus
$\a_2$ times the second slice ... plus $\a_{\aaa}$ times the $\aaa$-th slice.)
One may recover $T$ up to isomorphism
from the space of linear maps $T (A^*)$.

One says $T$ has {\it rank one} if
$T=a\ot b\ot c$ for some $a\in A$,
$b\in B$ and $c\in C$,  and the {\it rank} of $T$, denoted $\bold R(T)$ is
the smallest $r$ such that $T$ may be expressed as the sum of $r$ rank one
tensors. Rank is not semi-continuous, so one defines the {\it border rank} of
$T$, denoted
$\ur(T)$, to be the smallest $r$ such that $T$ is a limit of tensors of rank
$r$, or equivalently
(see e.g.~\cite[Cor. 5.1.1.5]{MR2865915}) the smallest $r$ such that $T$ lies in
the Zariski closure of the set of tensors of
rank $r$. Write ${\hat\s_r}(Seg(\BP A\times \BP B\times \BP C))\subset A
\ot B\ot C$ for
the variety of tensors of border rank at most $r$ (the cone over the $r$-th secant variety of the
Segre variety).
We will be mostly concerned with the case $\aaa=\bbb=\ccc=m$.

To make the connection with spaces of commuting matrices, we need to have
linear maps from a vector space to itself. Define $T\in A\ot B\ot C$
to be {\it  $1_A$-generic} if there exists $\a \in A^*$ with $T(\a)$ invertible.
Then $T(A^*)T(\a)\inv\subset \tend(B)$ will be our space of endomorphisms and
Strassen's equations for border rank $m$ is that this space is {\it abelian},
i.e., in bases
we obtain a space of commuting matrices.

Of particular interest is when an $m$-dimensional space of commuting matrices,
viewed as a point of the Grassmannian $G(m,\tend(B))$, is in the closure of
the space of diagonalizable subspaces (i.e., the maximal tori in $\fgl_n$),
which is denoted $Red(m)$ in \cite{MR2202260}.  Much of this paper will
utilize the interplay between the tensor and   endomorphism  perspectives.

Our primary 
  motivation  for this paper comes from the study of the complexity
of the matrix multiplication tensor $\Mn\in \BC^{n^2}\ot \BC^{n^2}\ot
\BC^{n^2}$.
We initiate  a geometric study of the tensors
used to prove upper bounds on the exponent of matrix multiplication,
especially the Coppersmith-Winograd tensor. 
In   \cite{DBLP:journals/corr/AmbainisFG14} they showed that one cannot
prove the exponent of matrix multiplication is less than exponent $2.3$ using the laser method
applied to the Coppersmith-Winograd tensor that was used for the current world record in
\cite{williams,stothers,legall}. The authors 
  suggested that to improve the upper bound on the exponent  one should look for other tensors
that give even better upper bounds via the   laser method. While the tensors of Strassen and Coppersmith-Winograd
were defined in terms of their {\it combinatorial} properties, we thought it would be useful
to isolate their {\it geometric} properties, and use these geometric properties as a basis for the
search. One geometric property is that they have (near) minimal border rank and relatively large 
rank. In this paper we find other tensors with the same property and 
we hope to investigate their {\it value} (in the sense of
\cite{williams,stothers,legall,DBLP:journals/corr/AmbainisFG14}) in future work.
On the other hand, to our surprise, {\it we isolate  two further geometric
properties that essentially   characterize  the Coppersmith-Winograd tensors}, see Theorems
\ref{cwflag}
and \ref{cwcompr}, which hints that one might have already  reached the limits of the laser method.

Another motivation from computer science is the construction
of {\it explicit} tensors of high rank and border rank, see, e.g.,
\cite{MR3025382, MR3144910}.
We give several such examples.

Our results include 
\begin{itemize}

\item Two purely geometric characterizations of the Coppersmith-Winograd tensor
(Theorems \ref{cwflag} and \ref{cwcompr}). 
\item Determination of the ranks of numerous tensors of minimal border rank including: all $1_A$-generic tensors that satisfy Strassen's equations for $m=4$ and $m=5$.
\item Proof, in \S\ref{brkmsect}, when $m\leq 4$, of a conjecture of J.~Rhodes
\cite[Conjecture 0]{MR3069955} for tensors in $\BC^m\ot \BC^m\ot \BC^m$ that
their ranks cannot be twice their border ranks, and counter-examples for all $m>4$ (Proposition \ref{prop:brank5}).
\item Explicit examples of tensors with rank to border rank ratio greater than two (Proposition \ref{biggapprop}
and Theorem \ref{biggergapthm}).
\item Proof that  the flag algebras of \cite{MR2202260} are of minimal border rank,
 \S\ref{flagalgsect}.
\item  Explicit examples showing that the known necessary conditions
for minimal border rank are independent.  

\item $1$-generic tensors satisfying Strassen's equations,  but 
 far from minimal border rank, \S\ref{abelovermsect}.

\item Proof that  $1$-generic tensors satisfying Strassen's
equations
must be symmetric (Proposition \ref{surpriseprop}).

\item New necessary conditions for border rank to be minimal  (Theorem
\ref{infflagprop})
with an example  (Example \ref{schemefails}), answering a 
question of A.~Leitner  \cite{MR3503693}.

\item A class of tensors for which Strassen's additivity conjecture holds  (Theorem \ref{addthm}).
\end{itemize}

\subsection{Background and previous work}
The maximum
rank of $T\in \BC^m\ot \BC^m\ot \BC^m$ is not known, it is easily seen to be
at most $m^2$ (and known to be at most $\frac 23 m^2$ \cite{MR3368091}), and of course is at least the maximum  border rank. The maximum
border rank is $\lceil\frac{m^3-1}{3m-2}\rceil$   except when
$m=3$ when it is five \cite{MR87f:15017,Strassen505}. In computer science, there is interest in
producing explicit tensors of high rank and
border rank.  The maximal   rank of a known  explicit tensor is
$3m-\tlog_2(m)-3$ when $m$ is a power of two
\cite{MR3025382}, see
Example \ref{biggestrsofar}.

Tensors in $A\ot B\ot C$ are completely understood when all vector
spaces have dimension at most three \cite{MR3239293}. In particular,
for tensors of border rank three, the maximum rank is five.
The case of $\BC^2\ot \BC^m\ot \BC^m$ is also completely understood, see, e.g.~\cite[\S 10.3]{MR2865915}.
While tensors of border rank $4$ in $\BC^4\ot\BC^4\ot \BC^4$ are essentially
understood \cite{MR2996364,FriedlandGross,MR2836258}, the ranks of such tensors
are not known. We determine their ranks under our two genericity hypotheses.
The  difficulty of  understanding border rank four tensors in $\BC^4\ot \BC^4\ot
\BC^4$ (which was first overcome in \cite{MR2996364})  was non-concision, which
we avoid in this paper.

Let  $Red^{0,SL}(m)$ denote the set  of all maximal tori in $SL(m)$, i.e., the
set of all $(m-1)$-dimensional abelian subgroups that
are diagonalizable. It can be given
a topology  (called the {\it Chabauty topology}, see \cite{MR3503693}) and its
closure
$Red^{SL}(m)$ is studied in \cite{MR3503693}. (A.~Leitner works over
$\BR$, but
this changes little.) If one considers the corresponding Lie algebras, one
obtains a subvariety of the Grassmannian $Red^{\fsl}(m)=Red(m) \subset
G(m-1,\fsl_m)$
that was studied classically, and is called the {\it variety of reductions} in
\cite{MR2202260}. More precisely,
$$
Red(m)=\ol{\{ E\in G(m-1,\fsl_m) \mid {\rm a\ basis \ of \ }E{\rm \ is\ simultaneously \ diagonalizable} \} }.
$$
One can equivalently prove results at the Lie group
or Lie algebra level.  As the Cartan subalgebras of $\fsl_m$ and $\fgl_m$ can be identified, by adding and dividing out by identity, we may equivalently work in the Grassmannian $G(m,\fgl_m)$.  

We   present the results of \cite{MR3503693} (some of which we had
found independently) in  tensor language  for the benefit of the tensor
community.

A tensor $T\in A\ot B\ot C$ is  {\it $A$-concise}
if the map
$T : A^*\ra B\ot C$ is injective,
and it is {\it concise} if it is
$A$, $B$ and $C$ concise. Equivalently, $T$ is $A$-concise
if it does not lie in
any $A'\ot B\ot C$ with   $A'\subsetneq  A$.
Note that if $T$ is $A$-concise, then $\ur(T)\geq \aaa$.

\begin{definition}\label{def:concisegeneric}
If $\bbb=\ccc=m$   define $T\in A\ot B\ot C$ to be {\it $1_A$-generic}  if
$T (A^*)$ contains an element of rank $m$. Define
$1_B,1_C$ genericity similarly and say $T$ is {\it $1$-generic}
if it is $1_A,1_B$ and $1_C$-generic.
\end{definition}

Note that if $T$ is $1_A$-generic, then $T$  is $B$ and $C$ concise and in
particular,  $\ur(T)\geq m$.

\subsection{Organization}
In \S\ref{rkmeqnsssect} we describe  necessary conditions for $1_A$-generic
tensors to have
border rank $m$. In addition to Strassen's equations, there is
an $\tend$-closed condition,   flag genericity conditions, and infinitesimal
flag
genericity conditions, the last of which is new. In \S\ref{aftmeth}, we describe
the method of
\cite{MR3025382} for proving lower bounds on the
ranks of explicit tensors.
This method has a consequence for the study of Strassen's additivity conjecture
that we describe
in \S\ref{strremsect}.  In \S\ref{abelovermsect} we study $1_A$-generic tensors
satisfying Strassen's equations that
have border rank greater than $m$, giving explicit examples where each of the
necessary conditions fail and showing that
such tensors can have very large border rank. Moreover, we show that a $1$-generic
tensor satisfying Strassen's equations is isomorphic to a symmetric tensor.
In \S\ref{brkmsect} we study $1_A$-generic tensors of minimal border rank,
presenting a sufficient condition
to have minimal border rank,   classifications when $m=4,5$, computing the
ranks as well, and explicit examples
of   tensors with large gaps between rank and border rank.
We conclude in \S\ref{cwvaluesect} with a geometric analysis of tensors that
have been useful for proving upper bounds on the complexity of the matrix multiplication
tensor, in particular, giving two geometric characterizations of the Coppersmith-Winograd tensor.

\subsection{Notation}
Let $V$ be a complex vector space, $V^*=\{\a: V\ra \BC\mid \a{\rm \ is\
linear}\}$ denotes the dual vector space,
$V^{\ot k}$ denotes the $k$-th tensor power,
$S^kV$ denotes the symmetric tensors in $V^{\ot k}$, equivalently, the
homogeneous polynomials
of degree $k$ on $V^*$, and $\La k V$ denotes the skew-symmetric tensors in $V^{\ot
k}$. If $U\subset V$ (resp.  $v\in V$), we let $U\upperp \subset V^*$ (resp.
$v\upperp\subset V^*$) denote its annhilator.

  Projective space  is $\BP V= (V\backslash 0)/\BC^*$. For $v\in V$, $[v]\in \BP
V$ denotes the corresponding point in projective
 space  and for any subset $Z\subset \BP V$, $\hat Z\subset V$ is the corresponding
cone in $V$.
  For a variety $X\subset \BP V$,    $X_{smooth}$ denotes its
smooth points.     For $x\in X_{smooth}$, $\hat T_xX\subset V$ denotes its
affine tangent space.
For a subset $Z\subset V$ or $Z\subset \BP V$, its Zariski closure is denoted
$\ol{Z}$.

 The irreducible polynomial representations
of $GL(V)$ are indexed by partitions $\pi=(p_1\hd p_q)$ with at most $\tdim V$
parts.
Let $\ell(\pi)$ denote the number of parts of $\pi$ (so $\ell((p_1\hd p_q))=q$), and let $S_{\pi}V$ denote
the irreducible $GL(V)$-module
corresponding to $\pi$. The conjugate partition to $\pi$ is denoted $\pi'$.

Since we lack systematic methods to prove   bounds on the ranks of tensors,
we often
rely on the presentation of a tensor in a given basis to help us.
For example, the structure tensor for the group algebra of  $\BZ_m$ in the
standard basis looks like
$$
M_{\BC[\BZ_m]}(A^*)=\large\{\begin{pmatrix}
x_0 & x_1 &\cdots &x_{m-1}\\
x_{m-1} & x_0 &x_1 &\cdots\\
\vdots  &   &\ddots  & \\
x_{1} & x_2 & \cdots & x_0
\end{pmatrix} \mid x_j\in \BC\large\}
$$
but after a change of basis (the discrete Fourier transform), it becomes
diagonalized so in the new basis
it is transparently of rank and border rank $m$.

For $T\in A\ot B\ot C$, introduce the notation
for $T(A^*)$ omitting the $x_j\in \BC$, e.g., for $M_{\BC[\BZ_m]}(A^*)$,
we write
\be\label{gpalg}
M_{\BC[\BZ_m]}(A^*)= \begin{pmatrix}
x_0 & x_1 &\cdots &x_{m-1}\\
x_{m-1} & x_0 &x_1 &\cdots\\
\vdots  &   &\ddots  & \\
x_{1} & x_2 & \cdots & x_0
\end{pmatrix}  .
\ene

For tensors $T,T'\in \BC^m\ot \BC^m\ot \BC^m=A\ot B\ot C$, we will say $T$ and $T'$ are {\it strictly isomorphic} if
there exists $g\in GL(A)\times GL(B)\times GL(C)$ such that $g(T)=T'$, and we will say $T,T'$ are
{\it isomorphic} if there exists $g\in GL(A)\times GL(B)\times GL(C)$ and $\s\in \FS_3$ such that
$\s(g(T))=T'$.

\subsection{Acknowledgments} We thank
the Simons Institute for the Theory of Computing, UC Berkeley, for providing a
wonderful environment during the fall 2014 program
{\it Algorithms and Complexity in Algebraic Geometry} during which work  on  this
article began.
We also thank L.~Manivel for useful discussions and pointing out the reference
\cite{MR3503693}, and the anonymous referee who gave many
useful suggestions. Michalek is a member of AGATES group and a PRIME DAAD fellow.

\section{Border rank $m$ equations for tensors in $\BC^m\ot \BC^m\ot
\BC^m$}\label{rkmeqnsssect}

\subsection{Strassen's commutativity equations}\label{streqnssect}
Throughout this sub-section $\tdim B=\tdim C=m$.
Given $T\in A\ot B\ot C$ and  $\a \in A^*$, consider  $T(\a)\in B\ot
C=\thom(C^*,B)$.
If  $T(\a)$ is invertible, for all
$\a'\in A^*$, we may consider
$T(\a')T(\a)\inv :B\ra B$.
Let $[X,Y]=XY-YX$ denote the commutator of the matrices $X,Y$.

Strassen's equations \cite{Strassen505} are:
for all $\a,\a_1,\a_2\in A^*$ with $T(\a)$ invertible,,
$$\trank[T(\a_1)T(\a)\inv, T(\a_2)T(\a)\inv]\leq 2(\ur(T)-m).
$$
In particular, if $\ur(T)=m$, then the space $T(A^*)T(\a)\inv\subset \tend(B)$ is abelian.
It is also useful to use Ottaviani's formulation of Strassen's equations
\cite{ottrento}:
consider the map
\begin{align*}
T_{A}^{\ww} :B^*\ot A&\ra \La 2 A\ot C\\
\b\ot a &\mapsto a\ww T(\b).
\end{align*}
If $\tdim A=3$,  $\ur(T)\geq \frac 12\trank(T_{A}^{\ww})$.
If one  restricts $T$ to a $3$-dimensional subspace of $A^*$,
the same conclusion holds.
In general $\trank (T_{A}^{\ww})\leq (\aaa -1)\ur(T)$, because for
a rank one tensor $a\ot b\ot c$, $(a\ot b\ot c)_{A}^{\ww}(A\ot B^*)
=a\ww A\ot c$, i.e. $\trank (a\ot b\ot c)_{A}^{\ww}=\aaa-1$.

To deal with the case where $T(\a)$ is not invertible
in Strassen's formulation,   recall  that  a linear map
$f: B\ra C^*$ induces linear maps
$f^{\ww k}:\La k B\ra \La k C^*$, and  that  $\La{m-1}B\simeq B^*\ot \La mB$.
Thus  $f^{\ww (m-1)}:\La{m-1}B\ra \La{m-1}C^*$ may be identified with  (up to a
fixed choice of scale)
a linear map $B^*\ra C$, and thus its transpose may be identified with a linear
map
$C^*\ra B$. If $f$ is invertible, this linear map coincides up to scale  with
the inverse. In bases it
is given by the cofactor matrix of $f$.
So to obtain polynomials,  
use  $(T(\a)^{\ww m-1})^T: \La{m-1}C \ra \La{m-1}B^*$ in place of $T(\a)\inv$  by
identifying
$\La{m-1}C \simeq C^*$, $\La{m-1}B^*\simeq B$, see \cite[\S 3.8.4]{MR2865915}
for details.

As a module, as observed in  \cite{LMsecb}, Strassen's degree $m+1$
$A$-equations are
\be\label{strmod}
 S_{m-1,1,1}A^*\ot S_{2,1^{m-1}}B^*\ot S_{2,1^{m-1}}C^*.
\ene

\subsection{The flag condition}
Much of this paper will use the fact that $T\in A\ot B\ot C$ may be recovered up
to strict isomorphism  from
the linear space $T(A^*)\subset B\ot C$, and if $\tdim B=\tdim C$ and there exists $\a\in A^*$ with $T(\a)$ invertible,
$T$ may be recovered up to isomorphism from the space $T(A^*)T(\a)\inv\subset \tend(B)$.  In this regard, we recall:

\begin{proposition} \cite[Cor. 2.2]{MR2865915}\label{prop:tensorranktomatrix}
 There exist $r$ rank one
elements of $B\ot C$ such that $T(A^*)$ is contained in their span if and only
if $\bold R(T)\leq r$. Similarly, $\ur(T)\leq r$ if and only if there exists a
  curve
$E_t$ in the Grassmannian  $G(r,B\ot C)$, where for $t\neq 0$, $E_t$ is spanned by
$r$ rank one elements and $T(A^*)\subset  E_0$  (which is
defined by the compactness of the Grassmannian).
\end{proposition}

The following two results  appeared in \cite[Ex. 15.14]{BCS} and \cite[Cor.~18]{MR3503693}:

\begin{corollary}\label{hitsseg}  Let $\aaa=m$ and let
$T\in A\ot B\ot C$ be $A$-concise. Then
$\ur(T)=m$ implies that $T(A^*)\cap Seg(\BP B\times \BP C)\neq\emptyset$.
\end{corollary}
\begin{proof}
If   $E_t\in G(r,B\ot C)$ is spanned by rank one elements for all $t\neq 0$,
then when $t=0$, it must contain at least one rank one element.
But since $\aaa=m$, $T(A^*)=E_0$.
\end{proof}


\begin{corollary} \label{flagcor}
  Let $T\in   A\ot B\ot C$ with $\aaa=m$ be $A$-concise.  If $\ur(T)=m$, then
  there exists a complete flag $A_1\subset \cdots A_{m-1}\subset A_m= A^* $,
with $\tdim A_j=j$,  such that
$\BP T(A_j)\subset  \s_j(Seg(\BP B\times \BP C))$.
\end{corollary}
\begin{proof}
Write $T(A^*)=\tlim_{t\ra 0}\tspan\{ X_1(t)\hd X_m(t)\}$ where $X_j(t)\in B\ot C$ have
rank one.
Then take $\BP A_k=\BP \tlim_{t\ra 0}\tspan\{ X_1(t)\hd X_k(t)\}\subset \BP T(A^*)$.
Since $\BP \{ X_1(t)\hd X_k(t)\}\subset\s_k(Seg(\BP B\times \BP C))$ the same
must be true in the limit,
and each limit must have the correct dimension because $\tdim \tlim_{t\ra 0}
\tspan \{
X_1(t)\hd X_m(t)\}=m$.
\end{proof}

Call the implication of Corollary \ref{flagcor} the {\it flag condition}.

  There are infinitesimal and scheme-theoretic analogs of Corollary
\ref{flagcor}, but we were unable to state them in general in a useful manner.
  Here is a special case that indicates the general case. For another example,
see \S\ref{flagexsect}. For a variety $X\subset \BP V$, and a smooth point $x\in
X$,
  $\hat T_xX\subset V$ denotes its affine tangent space.

  \begin{proposition}
 \label{infflagprop} Let $T\in  A\ot B\ot C$ with $\aaa=m$  be $A$-concise. If
$\ur(T)=m$
and $T(A^*)\cap Seg(\BP B\times \BP C)=[X_0]$ is a single point, then
$\BP (T(A^*)\cap \hat T_{[X_0]}Seg(\BP B\times \BP C))$ must contain
a line.
\end{proposition}

Call the implication of Proposition \ref{infflagprop} the {\it infinitesimal flag condition}.

\begin{proof}
Say $T (A^*)$ were the limit of $\tspan\{ X_{1 }(t)\hd X_{m }(t)\}$ with each $X_{j }(t)$ of rank one.
Then since $\BP T(A^*)\cap Seg(\BP B\times \BP C)=[X_0]$,
we must have each $X_{j }(t)$ limiting to $X_0$. But then
$\tlim_{t\ra 0} \tspan\{X_{1 }(t), X_{2 }(t)\}$, which must be two-dimensional,
must be contained in  $\hat T_{[X_0]}Seg(\BP B\times \BP C)$
and $T(A^*)$.  \end{proof}

\begin{remark} Because these conditions deal with intersections,
they are difficult to write down as polynomials. We will use them for tensors
with simple expressions where they can be checked.
\end{remark}

\subsection{Review and clarification of results in \cite{LMsecb}}
Throughout this subsection we assume $\bbb=\ccc=m$.
To a $1_A$-generic tensor $T\in A\ot B\ot C$, fixing $\alpha_0\in A^*$ as in
Definition
\ref{def:concisegeneric},   associate a subspace of endomorphisms of $B$:
$$\cE_{\alpha_0}(T):=\{T(\alpha)  T(\alpha_0)^{-1} \mid \alpha\in A^*\}\subset
\tend(B).$$

Note that $T$ may be recovered up to isomorphism from $\cE_{\a_0}(T)$.

\begin{lemma}\label{lem:com} Let $T\in A\ot B\ot C$ be $1_A$-generic and assume $\trank (
T(\a_0))=m$.
\begin{enumerate}
\item If   $\ur(T)=m$ then $\cE_{\alpha_0}(T)$ is commutative.
\item If $\cE_{\alpha_0}(T)$ is commutative  then $\cE_{\alpha_0'}(T)$ is
commutative for any $\alpha_0'\in A^*$ such that $\trank (T(\alpha_0'))=m$.
\end{enumerate}
\end{lemma}
\begin{proof} The first assertion is just a restatement of
Strassen's equations. For the second,
say $\cE_{\alpha_0}(T)$ is commutative, so
\be\label{havethis}
T(\a_1)T(\a_0)\inv T(\a_2)=T(\a_2)T(\a_0)\inv T(\a_1)
\ene
for all $\a_1,\a_2\in A^*$.
We need to show that
\be\label{wantthis}
T(\a_1)T(\a_0')\inv T(\a_2)=T(\a_2)T(\a_0')\inv T(\a_1)
\ene
for all $\a_1,\a_2\in A^*$.
Since $\cE_{\alpha_0}(T)$ is commutative, we have
$T(\a_0')T(\a_0)\inv T(\a_2)=T(\a_2)T(\a_0)\inv T(\a_0')$
and $T(\a_0')T(\a_0)\inv T(\a_1)=T(\a_1)T(\a_0)\inv T(\a_0')$, i.e., assuming
$T(\a_1),T(\a_2)$ are invertible,
$$T(\a_0)\inv = T(\a_0')\inv T(\a_j)T(\a_0)\inv T(\a_0')T(\a_j)\inv\ \ {\rm for\
\ }   j=1,2.
$$
Substituting the $j=2$ case to the left hand side of \eqref{havethis} and the
$j=1$ case
to the right hand side yields  \eqref{wantthis}. The cases where $T(\a_j)$ are
not invertible
follow by taking limits, as the $\a$ with $T(\a)$ invertible form a Zariski
open  subset of $T(A^*)$.
\end{proof}

If $U\subset B^*\ot B$ is commutative, then we may consider it as
an abelian Lie-subalgebra of $\fgl(B)$.

Define
\begin{align*}Abel_A:&=\BP \{ T\in A\ot B\ot C\mid T{\rm\ is\ } A-{\rm
concise,\
  }\\
&\ \ \exists \a_0\in A^*   \ {\rm with \ }
\trank(T(\a_0))=m, {\rm \ and \ }  \cE_{\a_0}(T)\subset \fgl(B) {\rm\ is \ an  \
abelian\ Lie\
algebra}\}
\\
&=\BP \{ T\in A\ot B\ot C\mid T{\rm\ is\ } A-{\rm  concise},\ 1_A-\rm{generic \ and\ }
\\
&\ \ \forall \a\in
A^*  \ {\rm with} \
\trank(T(\a))=m,  \cE_{\a}(T)\subset \fgl(B) {\rm\ is\ an \ abelian\ Lie\
algebra}\}
\end{align*}

The second  equality  follows from Lemma \ref{lem:com}(2).

\begin{definition}\label{abeltens} We say $T\in A\ot B\ot C$ is an {\it
$A$-abelian tensor}
if $T\in Abel_A$.
\end{definition}

$Abel_A$ is a Zariski closed subset of the set of concise $1_A$-generic tensors,
namely
the zero set of Strassen's equations. Its closure in $A\ot B\ot C$ is a
component of the zero set of
Strassen's equations.

Define
\begin{align*}
Diag_A^0:&=\BP \{ T\in A\ot B\ot C\mid T{\rm\ is\ } A-{\rm concise, \
}\\
&\ \ \ \ \exists \a_0\in A^* \ {\rm with \ }
\trank(T(\a_{ 0}))=m,  {\rm \  and \ } \cE_{\a_0}(T)\subset \fgl(B){\rm \ is\
diagonalizable}\},
\\
&=\BP \{ T\in A\ot B\ot C\mid T{\rm\ is\ } A-{\rm concise},\ 1_A-\rm{generic \ and\ }\\
&\  \ \ \ \forall \a\in
A^*  \ {\rm with \ }
\trank(T(\a))=m,  \ \cE_\a(T)\subset \fgl(B){\rm \ is\
diagonalizable}\}.
\end{align*}

Let $Diag_A$ be the Zariski closure of
  $Diag_A^0$ and let  $Diag_A^g$ be the intersection of $Diag_A$ with
  the set of concise $1_A$-generic tensors.

\begin{proposition}\label{thm:brankdiag} \cite{LMsecb} Let $A,B,C=\BC^m$, let
$T\in A\ot B\ot C$ be
concise  and   $1_A$-generic. Then the following are equivalent:
\begin{enumerate}
\item $\ur(T)=m$,

\item $T\in Diag_A^g$.
\end{enumerate}
Moreover, an abelian $m$-dimensional subspace of $\fgl(B)$ is in the closure of
the diagonalizable subspaces if and only if it arises as $\cE_{\a}(T)$ for  some
concise, border rank $m$
tensor $T\in A\ot B\ot C$.
\end{proposition}
 
\begin{proof}
Since the proof in the literature is not explicit and we use 
it frequently, we show that if $T\in A\ot B\ot C$ is concise, 
then $\cE_{\a}(T)$ belongs to the limit of diagonalizable subalgebras. We know there exists a 
curve of $m$-tuples of rank one tensors $(T_i^t)_{i=1}^m$ such that in the Grassmannian 
$$
\tlim_{t\rightarrow 0}\tspan\{T_i^t(A^*)\}\rightarrow  T(A^*).$$
In particular, there exist $X_t \in \tspan\{T_i^t(A^*)\}$, such that $X_t\rightarrow T(\alpha)$ and we may assume that $X_t$ 
are invertible. Then, $
\tspan\{T_i^t(A^*)\} X_t^{-1}$ is a curve of diagonalizable algebras converging to $\cE_{\a}(T)$.
\end{proof}

\subsection{The $\tend$-closed condition}
Throughout this subsection we assume $\bbb=\ccc=m$.
Define
$$ {End-Abel}_A:  =\BP \{ T\in Abel_A \mid \exists \a\in A^*
{\rm \ with \ } \trank(T(\a))=m,{\rm \ and\ }
\cE_{\a}(T) {\rm \ is \ closed \ under \ composition } \}.
$$

\begin{remark} There was   ambiguity in  the definition of $Comm_{\aaa,\bbb}$ in
\cite{LMsecb} that is clarified by the above notions which replace it.
\end{remark}

The following Proposition   essentially dates back to Gerstenhaber
\cite{MR0132079}. It is
utilized in \cite[\S 5]{MR3503693}   to obtain explicit abelian
subspaces that
are not in $Red(m)$, see \S\ref{abelalgvalg}.

\begin{proposition}\label{closedcor}
If $T\in A\ot B\ot C$, with $\bbb=\ccc=m$  is $1_A$-generic, $\ur(T)=m$ and
$\trank(T(\a_0))=m$, then
$\cE_{\alpha_0}(T)$ is closed under composition.
\end{proposition}
\begin{proof}
Each diagonalizable Lie algebra is closed under composition. The property
of being closed under composition is a Zariski closed property. \end{proof}

\begin{proposition}\label{cor:equal}
Let $T\in   End-Abel_A$ and
$\trank (T(\a_0))=m$.
Then $\cE_{\alpha_0}(T)=\cE_{\alpha_0'}(T)$ for any $\alpha_0'\in A^*$ such that $\trank (T(\alpha_0'))=m$.  In particular,
$$  End-Abel_A
 =\BP \{ T\in Abel_A \mid \forall \a\in A^* {\rm \ with \ } \trank(T(\a))=m, 
\cE_{\a}(T) {\rm \ is \ closed \ under \ composition } \}.
$$
\end{proposition} 
\begin{proof}
By the $\tend$-closed condition applied to $\a_0'$, for any $\alpha\in A^*$ there exists $\alpha'\in A^*$ such that 
$$
T(\a)T(\a_0')\inv T(\a_0)T(\a_0')\inv
=
T(\a')T(\a_0')\inv
$$
i.e.,
\begin{align*}
T(\a)T(\a_0')\inv 
&=
T(\a')T(\a_0')\inv T(\a_0')T(\a_0)\inv
\\
&=
T(\a') T(\a_0)\inv
\end{align*}
so 
$\cE_{\alpha_0'}(T)\subseteq \cE_{ {\alpha_0}}(T)$.  As both spaces are of the same dimension, equality must hold. 
\end{proof}

Note the inclusions
$$
Diag_A^g\subseteq End-Abel_A\subseteq Abel_A.
$$
  These spaces all coincide when $m\leq 4$,  $Diag_A^g=End-Abel_A$ for $m=5$, $End-Abel_A\subsetneq Abel_A$ when
$m\geq 5$ (see \S\ref{abelalgvalg}), and are all different when $m\geq 7$   (see \S\ref{diagvabelalg}).

\begin{proposition}
The subvariety $End-Abel_A$ in   the set of $1_A$-generic tensors has equations
that as
a $GL(A)\times GL(B)\times GL(C)$-module include
$$
S_{m,3,1^{m-2}}A^*\ot
(\bigoplus_{\stack{|\pi|=m+1}{p_1,\ell(\pi)\leq m}}S_{\pi+(1^m)}B^*\ot
S_{\pi'+(1^m)}C^*).
$$
These equations are of degree $2m+1$.
\end{proposition}
\begin{proof}
For
$\a_1\hd \a_m$ a basis of $A^*$,  if $T\in End-Abel_A$,  then
for all $\a,\a'\in A^*$,
$$T(\a ) (T(\a_1)^{\ww m-1})^TT(\a')\subset \tspan\{ T(\a_1)\hd T(\a_m)\}.
$$
In other words,   the following vector in $\La{m+1}(B\ot C)$  must be zero:
$$
T(\a ) (T(\a_1)^{\ww m-1})^T T(\a')
\ww T(\a_1)\ww\cdots \ww T(\a_m).
$$
The entries of this vector  are polynomials of degree $2m+1$ in the coefficients
of $T$, as the entries of
$(T(\a_1)^{\ww m-1})^T$ are of degree $m-1$ in the coefficients of $T$ and all
the other
matrices have entries that are linear in the coefficients of $T$.
Among the quantities that must be zero are the coefficients of
$b_1\ot c_1 \ww \cdots \ww b_1\ot c_m\ww b_2\ot c_1$, and more generally
the coefficients of
$b_{ 1}\ot c_{ 1}\ww\cdots \ww b_{ 1}\ot c_{q_{p_1}}\ww \cdots \ww b_{p_1}\ot
c_1\ww\cdots \ww b_{p_1}\ot c_{q_{p_1}}$ where
$\pi=(p_1\hd p_{q_1})$ is a partition of $m+1$ with first part at most $m$,
$q_1\leq m$,
and $\pi'=(q_1\hd q_{p_1})$. Now take $\a,\a'=\a_2$, the corresponding
coefficients
have the stated weight and all are
  highest weight vectors.
\end{proof}

 \section{The Alexeev-Forbes-Tsimerman method for bounding tensor
rank}\label{aftmeth}

Because the set of tensors of rank at most $r$ is not closed, there are few
techniques
for proving lower bounds on rank that are not just lower bounds for border rank.
What follows is the only general technique we are aware of.
(However for very special
 tensors like matrix multiplication, additional methods are available,
 see \cite{MR3162411}.)
The method   below, generally called the {\it
substitution method} was introduced in \cite{MR0207178} and used in \cite{MR0375839,MR0395328}
among other places.
We follow  the novel application of it from  \cite{MR3025382}.
 Fix a basis $a_1\hd a_{\aaa}$ of $A$.
Write  $T=\sum_{i=1}^\aaa a_i\otimes M_i$, where $M_i\in B\otimes C$.
\begin{proposition}\cite[Appendix
B]{MR3025382},  \cite[Chapter 6]{blaser2014explicit}\label{prop:slice}
Let $\bold R(T)=r$ and $M_1\neq 0$. Then there exist constants $\lambda_2,\dots,
\lambda_\aaa$, such that  the tensor
$$\tilde T:=\sum_{j=2}^{\aaa} a_j\otimes(M_j-\lambda_j M_1)\in  a_1 \upperp \ot B\ot
C,$$
has rank at most $r-1$. Moreover, if $\trank (M_1)=1$ then for any
choice  of  $(\lambda_2\hd \l_\aaa)$ we have $\bold R(\tilde T)\geq r-1$.
\end{proposition}
The statement of Proposition \ref{prop:slice} is   slightly different from the
original statement in
  \cite{MR3025382}, so we give
a modified   proof:
\begin{proof}
By Proposition \ref{prop:tensorranktomatrix}   there exist rank one
elements $X_1,\dots, X_r\in B\ot C$ and scalars $d^u_j$  such that:
$$
M_j=\sum_{u=1}^r d_j^u X_u.
$$
Since  $M_1\neq 0$ we may assume $d_1^1\neq  0$ and define
$\lambda_j=\frac{d_j^1}{d_1^1}$. Then the subspace
$\tilde T(  a_1\upperp  )$ is spanned by $X_2,\dots, X_r$ so
Proposition \ref{prop:tensorranktomatrix} implies $\bold R(\tilde T)\leq r-1$.
The last assertion holds because  if $\trank( M_1)=1$ then we may assume
$X_1=M_1$.
\end{proof}
  Proposition \ref{prop:slice} is usually implemented by consecutively applying
the following steps, which we will refer to as the substitution method:
\begin{enumerate}
\item Distinguish $A$, take a basis $\{ a_j\}$ of it  and
take   bases$\{\b_s\}$, $\{\g_t\}$  of $B^*,C^*$ and represent $T$ as a matrix
$M$ with entries that are linear combinations of the  basis vectors $a_i$:
$M_{s,t}=T(\beta_s\otimes\gamma_t)$.
\item Choose a subset of  $\bbb'$   columns
of $M$  and   $\ccc'$   rows of $M$.
\item Inductively, for elements of the chosen
columns (resp.~rows) remove the nonzero  $u$-th column (resp.~row) and add to all other
columns
(resp.~rows) the $u$-th column (resp.~row) times an arbitrary coefficient
$\lambda$,
regarding the $a_j$ as formal variables. This step is just  to ensure that each
time
only  nonzero columns or rows are removed.
\item Set all $a_j$ that appeared in any of the selected rows or columns to
zero, obtaining a matrix $M'$. Notice, that $M'$ does not depend on the choice
of $\lambda$.
\item The rank of $T$ is at least $\bbb'$ plus $\ccc'$ plus the rank of the
tensor
corresponding to  $M'$.
\end{enumerate}
The above steps can be iterated, interchanging the roles of $A,B$
and $C$.

\begin{example}\label{ex:aft} \cite{MR3025382}
Let
$$
T(A^*)=
\begin{pmatrix}
x_1 & & & & & & &\\
 &x_1 & & & & & &\\
  & & x_1& & & & &\\
   & & & x_1& & & &\\
    x_2& & & & x_1& & &\\
     & x_2& & & & x_1& &\\
      x_3&  & x_2& & & & x_1&\\
       x_4&x_3 & &x_2 & & & &x_1
\end{pmatrix},
$$
where here,  and in what follows, blank entries are zero.
Then $\bold R(T)\geq 15$.
Indeed, in the first iteration of the method presented above,   choose the first
four rows and last  four  columns. One  obtains a $4\times 4$ matrix $M'$ and the
associated tensor $T'$,
so    $\bold R(T)\geq 8+\bold R(T')$. Iterating the method twice yields
$\bold R(T)\geq 8+4+2+1=15$.

On the other hand $\ur(T)=8$, e.g., because $T(A^*)$, after
a choice of $\a^1$, is a specialization of a space that is abelian and contains
a regular nilpotent element and  we conclude by  Corollary \ref{containsreg} below.

To see that  $\bold R(T)=15$, one can construct an explicit expression or appeal
to Proposition \ref{centralizerprop} because
$T(A^*)$ is a degeneration   of the
centralizer of a regular nilpotent element.

This generalizes to $T\in \BC^k\ot \BC^{2^k}\ot \BC^{2^k}$ of rank
$2*2^k-1$ and border rank $2^k$.
\end{example}
\begin{example}
\cite{MR3025382}\label{biggestrsofar}
Let
$T=a_1\ot (b_1\ot c_1+\cdots + b_8\ot c_8)+a_2
\ot (b_1\ot c_5+b_2\ot c_6+ b_3\ot c_7+b_4\ot c_8)+
a_3\ot(b_1\ot c_7+b_2\ot c_8)+a_4\ot b_1\ot c_8+
a_5\ot b_8\ot c_1+a_6\ot b_8\ot c_2+a_7\ot b_8\ot c_3+
a_8\ot b_8\ot c_4$, so
$$
T(A^*)=
\begin{pmatrix}
x_1 & & & & & & &x_5\\
 &x_1 & & & & & &x_6\\
  & & x_1& & & & &x_7\\
   & & & x_1& & & &x_8\\
    x_2& & & & x_1& & &\\
     & x_2& & & & x_1& &\\
      x_3&  & x_2& & & & x_1&\\
       x_4&x_3 & &x_2 & & & &x_1
\end{pmatrix}.
$$
Then $\bold R(T)\geq 18$.
Here we start by  contracting
$x_5,x_6,x_7,x_8$. We obtain a tensor $\tilde T$ represented by the matrix
$$
\begin{pmatrix}
x_1 & & & & & & \\
 &x_1 & & & & & \\
  & & x_1& & & & \\
   & & & x_1& & & \\
    x_2& & & & x_1& &\\
     & x_2& & & & x_1&\\
      x_3&  & x_2& & & & x_1\\
       x_4&x_3 & &x_2 & & &
\end{pmatrix},
$$
and $\bold R(T)\geq 4+\bold R(\tilde T)$.
The substitution method  then gives  $\bold R(\tilde T)\geq 14$.
In fact, $\bold R(T)=18$; it is enough to consider $17$ matrices with just one
nonzero entry corresponding to all nonzero entries of $T(A^*)$, apart from the
top left and bottom right corner and one  matrix with $1$ at each corner and all
other entries equal to $0$.
This generalizes to $T\in \BC^{2^k}\ot \BC^{2^k}\ot \BC^{2^k}$ of rank
$3*2^k-k-3$.
\end{example}
For tensors in $\BC^m\ot \BC^m\ot \BC^m$, the limit of the method   would be to
prove a tensor has rank at least $3(m-1)$, and this
can be achieved only by exchanging the roles of $A,B,C$ in the application
successively.

\section{A Remark  on Strassen's additivity conjecture}\label{strremsect}

Strassen's additivity conjecture \cite{MR0521168} states that the
rank of the sum of two
tensors   in disjoint spaces equals the sum of the ranks.
While this conjecture has been studied from several different perspectives,
e.g.~ \cite{MR766508,MR861366,MR1617757,MR3320211,MR3092255},
very little is known about it, and experts are divided as to whether it should
be true or false.

  In many cases of low rank the substitution method provides the correct rank.
  In light of this, the following theorem indicates  why providing a
counter-example to Strassen's conjecture may be difficult.
\begin{theorem} \label{addthm} Let $T_1\in A_1\ot B_1\ot C_1$ and
$T_2\in A_2\ot B_2\ot C_2$ be such that
that $\bold R(T_1)$  can be determined  by the substitution method.
Then Strassen's additivity conjecture holds for
$T_1\op T_2$, i.e, $\bold R(T_1\op T_2)= \bold R(T_1)+ \bold R(T_2)$.
\end{theorem}
\begin{proof}
With each application of the substitution method, $T_1$ is modified to
a tensor of lower rank living in a smaller space and
  $T_2$ is unchanged. After all applications,  $T_1$ has been modified
  to zero and $T_2$ is still unchanged.
\end{proof}
  The rank of any tensor in $\BC^2\ot B\ot C$ can be computed using the substitution method 
  as follows:
  by dimension count, we can always find either $\beta \in B^*$ or $\gamma\in C^*$, 
  such that $T(\b)$ or $T(\g)$ is a rank one matrix. In particular, 
  Theorem \ref{addthm} provides an easy proof of Strassen's additivity conjecture if the dimension of any of 
$A_1,B_1$ or $C_1$
equals $2$. This was first shown in  \cite{MR861366} by other methods  and is further investigated by Buczy{\'n}ski and Postinghel, see   \cite{PostBucz}.

\section{Abelian tensors in $\BC^m\ot \BC^m\ot \BC^m$ with border rank
greater than $m$}\label{abelovermsect}

\subsection{$End-Abel_A\subsetneq Abel_A$ for $m\geq 5$}\label{abelalgvalg}

The lower bounds for border rank in the following two propositions appeared in \cite[Def.
16]{MR3503693} in the language of groups.
It answers \cite[Question A]{MR2202260} and
  provides an example asked for in
\cite[Remark after Question B]{MR2202260} and a concise $1$-generic tensor
$T\in \BC^5\ot\BC^5\ot \BC^5$ with $\ur(T)=6$ satisfying Strassen's equations.
\begin{proposition}\label{prop:Leitner}
Let $T_{Leit,5}=a_1\ot (b_1\ot c_1+ b_2\ot c_2+ b_3\ot c_3+ b_4\ot c_4+ b_5\ot
c_5)+a_2\ot(b_1\ot c_3+b_3\ot c_5)
+a_3\ot b_1\ot c_4 + a_4\ot b_2\ot c_4+ a_5\ot b_2\ot c_5$,
which gives rise to the linear space
$$
T_{Leit,5}(A^*)=\begin{pmatrix}
x_1 & & & & \\
0 &x_1 & & & \\
   x_2&0 & x_1& & \\
   x_3& x_4& 0& x_1& \\
 0   &x_5 &x_2 &0 & x_1\\
\end{pmatrix}.
$$
Then $\cE_{\a^1}(T_{Leit,5})$
is an abelian Lie algebra, but not $\tend$-closed. I.e.,
$T_{Leit,5}\in Abel_A$ but $T_{Leit,5}\not\in End-Abel_A$. In particular,
$T_{Leit,5}\not\in Diag_A$ so $\ur(T_{Leit,5})>5$. In fact, $\ur(T_{Leit,5})=6$
and $\bold
R(T_{Leit,5})=9$.
\end{proposition}
\begin{proof}
The first statements are verifiable by inspection. The fact that the border rank
of the tensor is at least $6$ follows from Theorem \ref{thm:brankdiag}. The fact
that border rank equals $6$ follows by considering rank one matrices:
$$
X_1=\begin{pmatrix}
 & & & & \\
 & & & & \\
  & & & & \\
   1& 1& & & \\
    1&1 & & & \\
\end{pmatrix},
X_2=\begin{pmatrix}
 & & & & \\
 & & & & \\
  \epsilon& & \epsilon^2& & \\
   & & & & \\
    1& &\epsilon & & \\
\end{pmatrix},
X_3=\begin{pmatrix}
\epsilon^2 & & & &\epsilon^4 \\
 & & & & \\
  & & & & \\
   & & & & \\
    1& & & & \epsilon^2\\
\end{pmatrix},$$$$
X_4=\begin{pmatrix}
 & & & & \\
 & & & & \\
  & \epsilon&\epsilon^2 & & \\
   & 1&\epsilon & & \\
    & & & & \\
\end{pmatrix},
X_5=\begin{pmatrix}
 & & & & \\
 & & & & \\
  & & & & \\
   1& & -\epsilon& \epsilon^2& \\
    & & & & \\
\end{pmatrix},
X_6\begin{pmatrix}
 & & & & \\
 & \epsilon^2& & & \\
  & -\epsilon& & & \\
   & & & & \\
    & 1& & & \\
\end{pmatrix}.$$
Then
$$
T_{Leit,5}=\tlim_{\ep\ra 0}
\frac 1{\ep^2} a_1\ot (-X_1 + X_3+X_4+X_5+X_6)+\frac 1\ep a_2\ot (X_2-X_3) +
a_3\ot X_5 + a_4\ot X_4+a_5\ot X_6,
$$
which is a sum of six rank one tensors.
In terms of tensor products,
\begin{align*}
T_{Leit,5}=\tlim_{\ep\ra 0}
\frac 1{\ep^2}
[&
-a_1\ot (b_1+b_2)\ot (c_4+c_5)
+ \ep a_2 \ot (b_1+\ep b_3)\ot (c_5+\ep c_3)\\
&+(a_1-\ep a_2)\ot (b_1+\ep^2b_5)\ot (c_5+\ep^2c_1)
+(a_1+\ep^2 a_4)\ot (b_2+\ep b_3)\ot (c_4+\ep c_3)\\
&+(a_1+\ep^2a_3)\ot (b_1-\ep b_3+\ep^2 b_4)\ot c_4
+ (a_1+\ep^2a_5)\ot b_2\ot (c_5-\ep c_3+\ep^2 c_2)
].
\end{align*}
 
The substitution method  shows that $\bold R(T_{Leit,5})\geq 9$. To prove equality,
 consider the $9$
rank $1$ matrices:
\begin{enumerate}
\item $3$ matrices with just one nonzero entry corresponding to $x_3,x_4,x_5$,
\item The six matrices 
$$
\begin{pmatrix}
1 & & & & \\
 & & & & \\
  -1& & & & \\
   & & & & \\
    -1& & & & \\
\end{pmatrix},
\begin{pmatrix}
& & & & \\
 & & & & \\
  & & & & \\
   & & & & \\
    & &-1 & &1 \\
\end{pmatrix},
\begin{pmatrix}
& & & & \\
 & & & & \\
  & & & & \\
   & & & 1& \\
    & & & & \\
\end{pmatrix},
\begin{pmatrix}
& & & & \\
 & 1& & & \\
  & & & & \\
   & & & & \\
    & & & & \\
\end{pmatrix},
\begin{pmatrix}
& & & & \\
 & & & & \\
  1& &1 & & \\
   & & & & \\
    1& &1 & & \\
\end{pmatrix},\begin{pmatrix}
& & & & \\
 & & & & \\
  1& &-1 & & \\
   & & & & \\
    -1& & 1& & \\
\end{pmatrix}.$$
\end{enumerate}
$T(A^*)$ is contained in the span of these matrices.
\end{proof}

Note that $T_{Leit,5}$ is neither $1_B$ nor $1_C$-generic.  The example easily
generalizes to higher $m$, e.g.~for $m=7$ we could take:
$$
T_{Leit,7}(A^*)=
\begin{pmatrix}
x_1& & & & & &\\
& x_1& & & & &\\
& &x_1 & & & &\\
x_2& & &x_1 & & &\\
x_3& & & &x_1 & &\\
x_4&x_7 & & & &x_1 &\\
&x_5 &x_6 &x_2 & & &x_1
\end{pmatrix}.
$$ 
 
\begin{proposition}
The following tensor:
$$
T_{Leit,6}'=a_1\ot(b_1\ot c_1+ \cdots + b_6\ot c_6) + a_2\ot (b_1\ot c_2+b_2\ot
c_3)
+a_3\ot b_1\ot c_5+ a_4\ot b_1\ot c_6 + a_5\ot b_4\ot c_5+a_6\ot b_4\ot c_6,
$$
which
gives rise to the abelian subspace:
$$
T_{Leit,6}'(A^*)=
\begin{pmatrix}
x_1 & & & & & \\
x_2 &x_1 &  & & & \\
  &x_2 &x_1 &  & & \\
     & & &x_1 & & \\
 x_3  & & &x_5 &x_1 & \\
 x_4   & & &x_6 & &x_1
\end{pmatrix},
$$
is not $\tend$-closed and
satisfies  $\ur(T_{Leit,6}')=7$ and $\bold R(T_{Leit,6}')=11$.
\end{proposition}
\begin{proof}
The border rank is at least $7$ as $T_{Leit,6}'(A^*)$ is not
$\tend$-closed. The rank is at least $11$ by the substitution method.

To prove that rank is indeed $11$ we notice that the $3\times 3$ upper square of
$T_{Leit,6}'(A^*)$
$$
\begin{pmatrix}
x_1 & &\\
x_2 &x_1 &\\
  &x_2 &x_1\\
\end{pmatrix}
$$
represents a tensor of rank at most $4$ by considering:
$$
\begin{pmatrix}
& &1\\
& &\\
  & &\\
\end{pmatrix},
\begin{pmatrix}
 & &\\
1 &1 &\\
1&1 &\\
\end{pmatrix},
\begin{pmatrix}
& &\\
1 &-1 &\\
-1&1 &\\
\end{pmatrix},
\begin{pmatrix}
1 & &-1\\
 & &\\
-1  & &1\\
\end{pmatrix}.
$$
Apart from  this square there are $7$ nonzero entries, so the rank is at most
$7+4=11$.

To compute the border rank notice that after removing the second row and column
we obtain a tensor of border rank $5$ by
Proposition \ref{prop:Young} below. On the other hand the entries in the second
column and row clearly form a border rank $2$ tensor. In other words, the tensor corresponding to
$$
\begin{pmatrix}
x_1 & & & & & \\
x_2 &x_7 &  & & & \\
  &x_2 &x_1 &  & & \\
     & & &x_1 & & \\
 x_3  & & &x_5 &x_1 & \\
 x_4   & & &x_6 & &x_1
\end{pmatrix}.
$$
has border rank $7$ and $T_{Leit,6}'$ is a specialization if it.
\end{proof}

\subsection{ $Diag^g_A\subsetneq End-Abel_A$ for $m\geq 7$ }\label{diagvabelalg}

\begin{proposition}
The tensor corresponding to 
$$
T_{end,7}(A^*)=
\begin{pmatrix}
x_1& & & & & &\\
& x_1& & & & &\\
& &x_1 & & & &\\
& & &x_1 & & &\\
&x_2+x_7 &x_3 &x_4 &x_1 & &\\
x_2&x_3 &x_5 &x_6 & &x_1 &\\
x_4&x_5 &x_6 &x_7 & & &x_1
\end{pmatrix}
$$
is $\tend$-closed, but has border rank at least $8$.
\end{proposition}
\begin{proof}
The fact that it is $\tend$-closed follows by inspection. The tensor has border rank at least $8$ by Corollary \ref{hitsseg} as $
T_{end,7}(A^*)$ does not intersect the Segre. Indeed, if it intersected Segre we would have $x_1=x_4=0$, and $(x_2+x_7)x_2=0$.  If $x_2= 0$, then $x_7^2=(x_2+x_7)x_7=0$, which implies $x_3=x_5=x_6=0$ and gives a contradiction. If $x_2+x_7= 0$ we obtain $x_3=x_5=x_6=0$. This implies $x_7=0$ and leads to a contradiction. \end{proof}

The following proposition yields families of  $End$-closed tensors of large  rank  for large $m$ 
and border rank greater
than $m$ as soon as $m\geq 8$:
\begin{proposition} Let $T\in A\ot B\ot C=\BC^m\ot \BC^m\ot \BC^m$
be such that
$T=a_1\ot (b_1\ot c_1+\cdots b_m\ot c_m)+ T'$
where $T'\in A'\ot B'\ot C':=\tspan\{ a_2\hd a_m\}
\ot \tspan\{ b_1\hd b_{\lfloor \frac m2\rfloor}\}
\ot \tspan\{ c_{\lceil \frac m2\rceil}\hd c_m\}$. Then $\bold R(T)= \bold R(T')+m$.
\end{proposition}
\begin{proof}
We have
$$T(A^*)\subset \begin{pmatrix}
x_1 & & & & & \\
& \ddots & & & & \\
& & x_1 & & & &\\
* & \cdots & * & x_1 & &\\
\vdots & \vdots & \vdots & & \ddots & \\
* &\cdots & * & & & x_1\end{pmatrix}.
$$
The lower bound on rank is obtained by  substitution method  and the
upper bound follows as rank is sub-additive.
\end{proof}

\begin{corollary}
A general element of $End-Abel_A$ has border rank at least $\frac{m^2}8$, which
is approximately
$\frac 38$-of the maximal  border  rank of a general tensor in $\BC^m\ot \BC^m\ot \BC^m$.
\end{corollary}
\begin{proof}
 $\ur(T)\geq \ur(T')\geq \frac{m^2}8$ for general $T'\in \BC^{m-1}\ot \BC^{\lfloor \frac m2\rfloor}
\ot \BC^{\lceil \frac m2\rceil}$.  
\end{proof}

\subsection{A generic $A$-abelian, $\tend$-closed tensor satisfying $\BP T(A^*)\cap Seg(\BP B\times \BP C)\neq\emptyset$  has high border rank}
\begin{proposition} Let $m-1=k+\ell$ with $3(k+\ell)<k\ell+5$. Then an abelian
tensor
$T\in \BC^m\ot \BC^m\ot \BC^m$
  of the form $T=\Mone\op T'$ with $\Mone\in \BC^1\op \BC^1\op \BC^1$
  and
$T'\in \BC^{m-1}\ot \BC^{m-1}\ot \BC^{m-1}=A'\ot B'\ot C'$ of the form
$a_1\ot(b_1\ot c_1+\cdots +b_{m-1}\ot c_{m-1})+T''$ with
$T''\in \BC^{m-2}\ot \BC^k\ot \BC^{\ell}$ general, where $\BC^k=\tspan\{b_1\hd
b_k\}$
and $\BC^\ell=\tspan\{c_k+1\hd c_{m-1}\}$, has
$\ur(T)> m $ despite being  $1_A$-generic, abelian,  $\tend$-closed,  and such
that
$\BP T(A^*)\cap Seg(\BP B\times \BP C)\neq\emptyset$.
\end{proposition}
\begin{proof}
 For general $T''$ the space $T''((\BC^{m-2})^*) \in \BC^k\ot \BC^{\ell}$  will not intersect
$\s_2(Seg(\pp{k-1}\times \pp{\ell-1}))$ if
$\tdim \BP T''(A^*)+ \tdim \s_2(Seg(\pp{k-1}\times \pp{\ell-1}))<\tdim \BP
(\BC^k\ot \BC^\ell)$,
i.e., if
$[2(k+\ell-2)-1]+[k+\ell-1]<k\ell-1$, i.e., if $3(k+\ell)<k\ell+5$, so
$T$ violates the flag condition.
\end{proof}

\subsection{A tensor in $End-Abel_8$ of border rank $>8$
via Corollary \ref{hitsseg}}
Consider (from \cite[Prop. 19]{MR3503693})
\begin{align*}
T_{Leit,8}''=&a_1\ot (b_1\ot c_1+\cdots + b_8\ot c_8)
+a_2\ot(b_4\ot c_5 +b_3\ot c_6)+a_3\ot( b_3\ot c_5+b_2\ot c_6)+
a_4\ot (b_2\ot c_5+ b_1\ot c_6)\\
&+a_5\ot (b_4\ot c_7+b_3\ot c_8)
+a_6\ot(b_3\ot c_7+b_2\ot c_8) +a_7\ot (b_2\ot c_7+b_1\ot c_8)
+a_8\ot (b_1\ot c_7+b_4\ot c_6)
\end{align*}
so
$$
T_{Leit,8}''(A^*)
=\begin{pmatrix}
x_1& & & & & & & \\
&x_1 &  & & & & & \\
& &x_1 & & & & & \\
& & &x_1  & & & & \\
&x_4 &x_3 &x_2 &x_1  & & & \\
x_4&x_3 &x_2 &x_8 & &x_1  & & \\
x_8& x_7&x_6  &x_5 & & &x_1  & \\
x_7&x_6 & x_5& & & & & x_1
\end{pmatrix}.
$$
Since $\BP T(A^*)\cap Seg(\BP B\times \BP C)=\emptyset$,
  $\ur(T_{Leit,8}'')>8$. In Leitner's language the bound arises because the
corresponding group
does not contain a one parameter subgroup of rank  one.

\subsection{$A$-abelian tensors that intersect  the
Segre and fail to satisfy flag condition}
 The following example answers a question about  sufficient conditions to be a limit of diagonalizable groups 
  in \cite[p.~10]{MR3503693}. 

\begin{example} Consider $T_{s9}:= \Mone\op T_{Leit,8}'' \in \BC^9\ot \BC^9\ot
\BC^9=A\ot B\ot C$.
To keep indices consistent with above, write $\Mone=a_0\ot b_0\ot c_0$, so they
range from $0$ to $8$.
Note that $T_{s9}\cap Seg(\pp 8\times \pp 8\times \pp 8)\neq \emptyset$, and
$T_{s9}$ is $1_A$-generic and abelian, however $\ur(T_{s9})>9$.
To see this, note that the flag condition fails because  there is
no $\pp 1$ contained in $\s_2(Seg(\pp 8\times \pp 8))$.
\end{example}

\subsection{An End-closed tensor satisfying the flag condition but not the
infinitesimal flag condition} \label{schemefails}

Consider $$
T_{flagok}(A^*):=
\begin{pmatrix}
x_1&&&&&&&&\\
x_0&x_1& & & & & & & \\
&&x_1 &  & & & & & \\
&& &x_1 & & & & & \\
&& & &x_1  & & & & \\
&&x_4 &x_3 &x_2 &x_1  & & & \\
x_4&&x_3 &x_2 &x_8 & &x_1  & & \\
x_8&& x_7&x_6  &x_5 & & &x_1  & \\
x_7&&x_6 & x_5& & & & & x_1
\end{pmatrix}
$$
  Here  $\BP T_{flagok}(A^*)\cap Seg(\pp 8\times \pp 8)=[X_0]$,
where $X_0$ is the matrix with $1$ in the $(2,1)$ entry and zero elsewhere,
and
$$
\hat  T_{[X_0]}Seg(\pp 8\times \pp 8)
=\begin{pmatrix}
* & 0 & \cdots & 0\\
* & * &\cdots & *\\
* & 0 & \cdots & 0\\
\vdots & & \vdots & \\
* & 0 &\cdots & 0
\end{pmatrix}
$$
which only intersects   $T_{flagok}(A^*)$ in $[X_0]$. Thus  by Proposition
\ref{infflagprop},  $\ur(T_{flagok})>9$.

 The  flag condition is satisfied:   consider   respectively spaces spanned by $x_0,x_4,x_3,x_2,x_8, x_7,x_6,x_{ 5}$. It straightforward to check that
$T_{flagok}$ is
  $\tend$-closed.

 \subsection{$1$-generic abelian tensors}

In general, the spaces   $T(A^*),T(B^*),T(C^*)$ can be very different,
e.g.~if $T$ is $1_A$-generic and not $1_B$-generic, or if $T$
is not abelian. The following proposition 
is a variant of remarks in \cite[\S 7.7.2]{MR2865915} and \cite{MR2996364}:

\begin{proposition}\label{surpriseprop} Let $T\in A\ot B\ot C=\BC^m\ot \BC^m\ot
\BC^m$ be $1_A$ and $1_B$ generic and satisfy the $A$-Strassen's equations.
Then $T$ is isomorphic to a tensor
in $S^2\BC^m\ot \BC^m$.

In particular:

\begin{enumerate}
\item After making choices of general $\a\in A^*$ and $\b\in B^*$, $T(A^*)$ and
$T(B^*)$ are $GL_m$-isomorphic   subspaces
of $\tend(\BC^m)$.
\item If $T$ is $1$-generic, then $T$ is isomorphic to a tensor in $S^3\BC^m$.
\end{enumerate}
\end{proposition}
\begin{proof} Let $\{a_i\},\{b_j\},\{c_j\}$ respectively be bases of $A,B,C$.
Write $T=\sum t^{ijk}a_i\ot b_j\ot c_k$. Possibly after a change of basis we may
assume
$t^{1jk}=\d_{jk}$ and $t^{i1k}=\d_{ik}$. Take $\{\a^i\}$ the dual basis to $\{a_j\}$ and
identify $T(A^*)\subset \tend(\BC^m)$ via $\a^1$. Strassen's $A$-equations then
say
$$
0=[T(\a^{i_1}),T(\a^{i_2})]_{ (j,k)} =\sum_l t^{i_1 jl}t^{i_2 lk}- t^{i_2 jl}t^{i_1 lk} \
\forall i_1,i_2,j,k.
$$
Consider when $j=1$:
$$
0=\sum_l t^{i_1 1l}t^{i_2 lk}- t^{i_2 1l}t^{i_1 lk}= t^{i_2 i_1 k}-  t^{i_1
i_2k}  \ \forall i_1,i_2, k,
$$
because $t^{i_1 1l}=\d_{i_1,l}$  and $t^{i_2 1l}=\d_{i_2,l}$. But this says $T\in S^2\BC^m\ot \BC^m$.

For the last assertion, say $L_B: B\ra A$ is such that
$Id_A\ot L_B\ot Id_C(T)\in S^2A\ot C$ and $L_C: C\ra A$ is
such that $Id_A\ot Id_B\ot L_C\in S^2A\ot B$. Then
$Id_A\ot L_B\ot L_C(T)$ is in $A^{\ot 3}$, symmetric in the first and second factors as well
as the first and third. But $\FS_3$ is generated by two transpositions, so
$Id_A\ot L_B\ot L_C(T)\in S^3A$.
\end{proof}

Thus   the $A,B,C$-Strassen's equations, despite being very different modules,
when restricted to $1$-generic tensors,  all have the same zero sets.
Strassen's equations in the case of partially symmetric tensors were essentially
known to Emil Toeplitz \cite{MR1509924}, and
in the symmetric case to Aronhold \cite{MR1579064}.

\begin{proposition}  There exist $1$-generic abelian tensors $T\in  \BC^{4k}\ot \BC^{4k}\ot \BC^{4k}$ that have
border rank at least
  $ \frac{k(k+1)}{6} $. In particular,  there exist tensors
    in $\BC^{4k}\ot \BC^{4k}\ot \BC^{4k}$ satisfying the Strassen-Aronhold
  equations with $ \ur(T)\geq \frac{k(k+1)}{6} $.
\end{proposition}
 
 \begin{proof} We exhibit  a family of  $4k$-dimensional  subspaces of $
S^2\BC^{4k}$
 whose associated tensor
 can degenerate to a general tensor $T'\in S^2\BC^k \otimes \BC^{k-1}$.
 Since the border rank of a general tensor in $S^2\BC^k \otimes \BC^{k-1}\subset \BC^k\ot\BC^k\ot\BC^{k-1}$ must satisfy:
$$
\dim\s_{\ur(T')}(Seg(\BP^{k-1}\times\BP^{k-1}\times\BP^{k-2})) 
>\tdim \BP(S^2\BC^k \otimes \BC^{k-1}))=\frac{(k-1)k(k+1)}{2}-1,
$$
and $\dim\s_{r}(Seg(\BP^{k-1}\times\BP^{k-1}\times\BP^{k-2}))= r(3k-4)+r-1$
    (by \cite{MR87f:15017}),
we obtain:
$$\ur(T)\geq\ur(T')\geq \frac{k(k+1)}{6}.$$
 
Represent  $T$
  as a $4k\times 4k$ symmetric matrix $M$ with entries that are linear functions
  of  $4k$ variables. The variables will play four different
  roles, so we name them accordingly:
\begin{enumerate}
\item variable $z$,
\item variables $y_1,\dots,y_{k-1}$,
\item variables $x_1,\dots,x_k$,
\item variables $z_1,\dots,z_{2k}$.
\end{enumerate}
They appear in the given order (starting from the top) in the first column of
$M$,
ensuring $1_B$-genericity. This also defines the last row of $M$, ensuring
$1_C$-genericity. (Here we order the bases of $B$ and $C$ in opposite
order to have the matrix symmetric with respect to an anti-diagonal
reflection.)
The matrix $M$ will be lower triangular, with the variable $z$ on the diagonal
(ensuring $1_A$-genericity), and $z$  appears only on the diagonal.
The variables $z_1,\dots,z_{2k}$ appear only in the  last  row and  first  column as
defined previously.

Write a general tensor $T'\in S^2\BC^k \otimes \BC^{k-1}$ as $T'=\sum a_{ij}^l
e_i\otimes e_j\otimes f_l$, where $a_{ij}^l=a_{ji}^l$.
We use $T'$ to define the entries in rows $3k+1,\dots,4k-1$ and columns
$k+1,\dots,2k$: let  $\sum_{i=1}^k a_{is}^u x_i$ be the entry in 
the $(3k+u,k+s)$-position. As $M$ is
symmetric this defines also the entries in rows $2k+1,\dots,3k$
and columns $2,\dots, k$.

Apart from entries defined so far, the only 
remaining nonzero entries   belong  to a
$k\times k$ submatrix of rows from $2k+1,\dots,3k$ and columns
$k+1,\dots,2k$: the linear form in the
$y$'s    $\sum_{i=1}^{k-1} a_{(k+1-u),s}^{k-i} y_i$
is the $(2k+u,k+s)$-entry.

For example, when  $k=3$,

$$T(A^*)=
\begin{pmatrix}
z&&&&&&&&&&&\\
y_1&z&&&&&&&&&&\\
y_2&&z&&&&&&&&&\\
x_1&&&z&&&&&&&&\\
x_2&&&&z&&&&&&&\\
x_3&&&&&z&&&&&&\\
z_1&\sum_i a_{3i}^2x_i&\sum_i a_{3i}^1x_i&\sum_l a^{3-l}_{13}y_l&\sum_l
a^{3-l}_{23}y_l&\sum_l a^{3-l}_{33}y_l&z&&&&&\\
z_2&\sum_i a_{2i}^2x_i&\sum_i a_{2i}^1x_i&\sum_l a^{3-l}_{12}y_l&\sum_l
a^{3-l}_{22}y_l&\sum_l a^{3-l}_{23}y_l&&z&&&&\\
z_3&\sum_i a_{1i}^2x_i &\sum_i a_{1i}^1x_i&\sum_l a^{3-l}_{11}y_l&\sum_l
a^{3-l}_{12}y_l&\sum_l a^{3-l}_{13}y_l&&&z&&&\\
z_4&&&\sum_i a_{1i}^1x_i&\sum_i a_{2i}^1x_i&\sum_i a_{3i}^1x_i&&&&z&&\\
z_5&&&\sum_i a_{1i}^2x_i&\sum_i a_{2i}^2x_i&\sum_i a_{3i}^2x_i&&&&&z&\\
z_6&z_5&z_4&z_3&z_2&z_1&x_3&x_2&x_1&y_2&y_1&z\\
\end{pmatrix}.
$$

To prove that Strassen's equations are satisfied, take a matrix $M$ in variables
as above and a
matrix $M'$ with primed variables, then it is sufficient to
show that
each entry of the product matrix $MM'$  is symmetric as a bilinear form.
This is obvious for:
\begin{enumerate}
\item a product of any of the first $2k$ rows of $M$ with any column of $M'$,
\item a product of any of the rows $2k+1,\dots,4k-1$ of $M$ with
any of the columns $2,\dots,4k$ of
$M'$,
\item a product of the last row of $M$ with column $1$ or any of the columns
$2k+1,\dots,4k$ of $M'$.
\end{enumerate}
As all matrices are symmetric it remains to check the assumption for the product
of rows $2k+1,\dots,4k-1$ of $M$ and the first column of $M'$.

For the row $3k+l$ we obtain the symmetric linear form $\sum_j(\sum_i a^l_{ji}x_i)x_j'$.

For the row $2k+n$ we obtain the symmetric linear form $\sum_{l=1}^{k-1} (\sum_i
a_{k+1-n,i}^{k-l}x_i)y_l'+\sum_{u=1}^k  (\sum_l a_{(k+1-n),u}^{k-l}y_l)x_u'$.\end{proof}

\section{Tensors of border rank $m$ in $\BC^m\ot\BC^m\ot \BC^m$}\label{brkmsect}

In this section we present sufficient conditions for tensors
to be of minimal border rank and determine the ranks of several
examples of minimal border rank.

\subsection{Centralizers of a regular element}
Let $A=B=C=\BC^m$.
An element  $x\in \tend(B)$ is {\it regular}
if $\tdim  \bold C(x)=m$, where
$\bold C(x):=\{ y\in \tend(B)\mid [x,y]=0\}$ is the {\it centralizer} of $x$. We
say
$x$ is {\it regular semi-simple}  if
$x$ is diagonalizable with distinct eigenvalues. Note that  $x$ is regular
semi-simple
if and only if
  $\bold C(x)$ is diagonalizable.

The following proposition was  communicated to us by L.~Manivel.

\begin{proposition}\label{prop:regular} Let $U\subset \tend(B)$ be an abelian
subspace
of dimension $m$.
If there exists $x\in U$ that is regular, then
$U$ lies in the Zariski closure of the diagonalizable $m$-planes
in $G(m,\tend(B))$, i.e., $U\in Red(m)$. More generally, if there exist
$x_1,x_2\in U$, such that $U$ is their common centralizer, then $U\in Red(m)$.

\end{proposition}
\begin{proof}
Since the Zariski closure of the set of  regular semi-simple   elements is all of $\tend
(B)$,
for any $x\in \tend(B)$,  there
exists a curve $x_t$ of regular semi-simple elements with $\tlim_{t\ra 0}
x_t=x$.
Consider the induced curve in the Grassmannian
$\bold C(x_t)\subset G(m,\tend(B))$. Then  $\bold C_0:=\tlim_{t\ra 0}\bold
C(x_t)$ exists and is
contained in $\bold C(x)\subset \tend(B)$ and since
$U$ is abelian, we also have $U\subseteq\bold  C(x)$. But if $x$ is regular,
then
$\tdim\bold  C_0=\tdim (U)=m$, so
$\tlim_{ t\ra 0}\bold C(x_t)$, $\bold C_0$ and $U$ must
be equal and thus $U$ is a limit of diagonalizable subspaces.

 The proof of the second statement is similar,
  as a pair of commuting matrices can be approximated by a pair of diagonalizable
commuting matrices and diagonalizable commuting matrices are simultaneously
diagonalizable, cf.~\cite[Prop.~4]{MR2202260}.
\end{proof}

\begin{corollary} \label{containsreg}Let $T\in A\ot \BC^m\ot \BC^m$ with $\tdim
A\leq m$ be
such that $T(A^*)$ contains an element of rank $m$ and, after using it
to embed $T(A^*)\subset \fgl_m$, it is abelian and contains a regular element.
Then $\ur(T)=m$.
\end{corollary}

\begin{proposition}\label{centralizerprop}
Let $T(A^*/{\BC}\a)\subset \fsl(B) $ be the centralizer of a regular element of
Jordan type
$(d_1\hd d_q)$. Then $\ur(T)=m$ and $\bold R(T) = 2(\sum_{j=1}^q  d_j) -q$. In
particular, if
$T(A^*)$ is the centralizer of a regular nilpotent element, then
$\bold R(T)= 2m-1$ and if it is the centralizer of a regular semi-simple element
then $\bold R(T)=m$.
\end{proposition}
\begin{proof}
It remains to show the second assertion.
The substitution method  gives  the lower bound on $\bold R(T)$ for the
regular nilpotent case, and
  Theorem \ref{addthm} gives the lower bound for the general case.
For the upper bound, it is sufficient to prove the regular nilpotent case.
The tensor $M_{\BC[\BZ_{2m-1}]}$ (see Equation \eqref{gpalg}) specializes to
$T$ as follows: consider the space     $M_{\BC[\BZ_{2m-1}]}(A^*)$, cut the first
$\lfloor \frac m2\rfloor$ rows and the last $\lfloor \frac m2\rfloor$ columns
and set all entries
appearing above the diagonal in the remaining matrix to zero.
E.g., when $m=2$,
$$\begin{pmatrix} x_1 & x_2 & x_3\\ x_3&x_1&x_2\\
x_2&x_3&x_1\end{pmatrix} \ra
\begin{pmatrix}   x_3&x_1 \\
x_2&x_3 \end{pmatrix} \ra \begin{pmatrix}   x_3&0\\
x_2&x_3 \end{pmatrix}.
$$
\end{proof}

\subsection{The flag algebras of \cite{MR2202260}}\label{flagalgsect}
 We answer a question posed in
\cite[p.~4/p.~5]{MR2202260} whether certain algebras derived from flags belong
to $Red(m)$.
We start by presenting these algebras. Using matrix notation, the algebras   are
given by a partition $\l$ of size $|\l|=m-1$, to which
we associate a
  Young tableau with entries $ \{x_2\hd x_m\} $ whose reflection (across a
vertical line for the American presentation
and across a diagonal line for the French presentation) we situate in the upper
right hand block of the  $m\times m$ matrix $T(A^*)$
and we fill the diagonal with $x_1$'s.
For example $\l=(4,2,1)$ gives rise to 
the following $8$-dimensional subspace of $\BC^8\ot \BC^8$:
\be\label{fourtwooneold}
T(A^*)=\begin{pmatrix}
 x_1& & & &x_2& x_3& x_4&x_5 \\
 & x_1& & & & &x_6 &x_7 \\
  & & x_1& & & & &x_8\\
   & & &x_1& & & & \\
    & & & &x_1& & & \\
     & & & & & x_1& &\\
     & & & & & &x_1& \\
      & & & & & & &x_1 \\
\end{pmatrix}.
\ene

Call this space the abelian Lie algebra associated to the flag induced
by $\l$, and the corresponding tensor in $\BC^m\ot \BC^m\ot \BC^m$
the tensor associated to the flag induced by $\l$.
\begin{proposition}\label{prop:Young}
The abelian Lie algebras associated to flags, defined in
\cite[p.~4/p.~5]{MR2202260} belong to
 $Red(m)$, the closure of the diagonalizable algebras.
 That is,  the associated tensors to these abelien Lie algebras
are of border rank $m$.
\end{proposition}
\begin{proof}
For the purposes of this proof, it will be
convenient to re-order bases such that the  $x_1$'s occur on the anti-diagonal,
and the Young
tableau occur with the American presentation:
\be\label{fourtwoone}
T(A^*)=\begin{pmatrix}
 x_2&x_3 & x_4&x_5 & &  &  & x_1 \\
 x_6& x_7& & & & &  x_1&  \\
 x_8 & &  & & & x_1& & \\
   & & & & x_1& & & \\
    & & & x_1& & & & \\
     & & x_1& & &  & &\\
     & x_1& & & & & & \\
      x_1& & & & & & &
\end{pmatrix}.
\ene

 Suppose that $T(A^*)$ is  defined by a
partition $\lambda=( k^{l_k}\hd 2^{l_2},1^{l_1})$ with $|\l|=m-1$.
We define $m$ rank $1$ matrices, parametrized by $\epsilon$ such that $T(A^*)$
equals the limit of the span of these matrices as $\epsilon\rightarrow 0$,
considered as a curve in the Grassmannian $G(m,\fgl(B))$. The   matrices will
belong to five
groups.

We label the rows and columns of our matrix by $0\hd m-1$.

1) The first group contains just one matrix with four nonzero entries:
$$
\begin{pmatrix}
1 &0&\dots&0&\epsilon^{m-1} \\
 & &\vdots  &&\\
 \epsilon^{m-1}& 0&\dots &0 &\epsilon^{2m-2}  \\
\end{pmatrix}.$$

2) The second group also contains   one matrix, with support contained in the
$  \ell(\l) \times k$ upper left  rectangle.   For $(i,j)$ in the  rectangle,
we fill the entry with $\epsilon^{i +j }$  and set all other entries to zero.
So the tensor \eqref{fourtwoone} gives
$$
\begin{pmatrix}
\ep^0 &\ep^1 &\ep^2 &\ep^3 & &  &  &  \\
 \ep ^1&\ep^2  &\ep^3 &\ep^4 & & &  &  \\
  \ep^2& \ep^3&\ep^4  &\ep^5 & & & &\\
   & & & & & & & \\
    & & & & & & & \\
     & & & & &  & &\\
     & & & & & & & \\
      & & & & & & &  \\
\end{pmatrix}.
$$

3) The third group contains $m -k-\ell(\l) $ matrices. Each matrix
corresponds to an entry in the Young diagram $\lambda$ that is neither in the
zero-th row or column.  Notice that the number of such entries equals the
number of anti-diagonal entries of the matrix  that are not in the first $k-1$
rows
or first $(\sum_{i=1}^k l_i)-1$ columns. Fix a bijection between them.    To each such entry of
the Young diagram we associate a rank one matrix with only four nonzero entries.
Suppose that the entry of the Young diagram is the $(i_0,j_0)$ entry of the
matrix  and  the corresponding entry of the anti-diagonal is $(i_1,j_1=m-1-i_1)$. The entries are
$$
a^i_j=\left\{
\begin{matrix}
\ep^{i_0+j_0} & (i,j)=(i_0,j_0)\\
\ep^{m-1} & (i,j)=(i_1,j_1) \\
 \ep^{i_1+j_0} & (i,j)=(i_1,j_0)\\
 \ep^{i_0+j_1} & (i,j)=(i_0,j_1)\\
0 & {\rm otherwise}
\end{matrix}\right\} . 
$$

The tensor \eqref{fourtwoone} has
$$
\begin{pmatrix}
  & & & & &  &  &  \\
  & \ep^2 & & &\ep^5 & &  &  \\
  & &  & & & & & \\
  &\ep^4 & & & \ep^7& & & \\
  & & &  & & &  & \\
     & & & & &  & &\\
     & & & & & & & \\
      & & & & & & &  \\
\end{pmatrix}.
$$
4) The fourth group contains $k -1$ matrices. These correspond to the
entries in the $0$-th  row  of $\lambda$ but not in the $0$-th  column. The matrix
corresponding to the entry in the $i$-th column is defined by 
$$
a^i_j=\left\{
\begin{matrix}
\ep^i & j=0\\
\ep^{i+j} & j\leq k-1{\rm\ and \ } (i,j)\not\in {\rm Young\ diagram\ of } \l\\
-\ep^{i+j} & a^{m-1-j}_j {\rm has\ been\ associated\ to \ }a^i_j\\
\ep^{m-1} & j=m-1-i\\
0 & {\rm otherwise}
\end{matrix}\right\}
$$
The tensor \eqref{fourtwoone} has
$$
 \begin{pmatrix}
  &\ep^1 & & & &  &  &  \\
 &  & & & & &  &  \\
  &\ep^3 &  & & & &  & \\
   &-\ep^4 & & & & & & \\
    & & & & & & & \\
     & & & & &  & &\\
     &\ep^7 & & & & &  & \\
      & & & & & & &  \\
\end{pmatrix}, \ \
\begin{pmatrix}
  & &\ep^2 & & &  & &  \\
 &  & \ep^3& & &  & &  \\
  & &\ep^4  & & &  & & \\
   & & & & & & & \\
    & & & & & & & \\
     & & \ep^7  & & &   & &\\
     & & & & & & & \\
      & & & & & & &  \\
\end{pmatrix},
 \ \
\begin{pmatrix}
  & & &  \ep^3& &  & &  \\
 &  & &   \ep^4& & &  &  \\
  & &  &  \ep^5& & & & \\
   & & & & & & & \\
    & & & \ep^7& & & & \\
     & & &  & &  & &\\
     & & & & & & & \\
      & & & & & & &  \\
\end{pmatrix}.
$$

5) The fifth group, consisting of $\ell(\l)-1$ matrices,  is analogous to the
fourth with entries corresponding to rows
instead of columns  and all entries, apart from the first  column, in the $  \ell(\l) \times k$ upper left  rectangle equal to zero.

The tensor \eqref{fourtwoone} has
$$
\begin{pmatrix}
  & & & & &  & &  \\
\ep^1 &  & & &-\ep^5  &  &   \ep^7& \\
  & &  & & & & & \\
   & & & & & & & \\
    & & & & & & & \\
     & & & & &  & &\\
     & & & & & & & \\
      & & & & & & &  \\
\end{pmatrix}, \ \
\begin{pmatrix}
  & & & & &  &  &  \\
 &  & & & & &  &  \\
 \ep^2&  & & & \ep^7& &  &  \\
   & & & & & & & \\
    & & & & & & & \\
     & & & & &  & &\\
     & & & & & & & \\
      & & & & & & &  \\
\end{pmatrix}.
$$
Except for the second group, each matrix in each group has a distinguished
element in the Young diagram which, after normalization, is the limit as
$\epsilon\rightarrow 0$.
Moreover, summing all matrices from groups $1,3,4,5$ and subtracting the matrix
from group $2$, the limit as $\epsilon\rightarrow 0$ is the identity matrix.
\end{proof}

\subsection{Case $m=4$}\label{brankfour}
Fix an $A$-concise, $1_A$-generic  border rank $4$ tensor $T\in A\otimes
B\otimes C$, where $A,B,C\simeq\BC^4$.
The contraction $T( A^*)$ is a $4$-dimensional subspace of matrices and we may
assume that it contains the identity.
By \cite[Prop. 18]{MR2202260} we may assume that the space $T( A^* )$ is one
of the $14$ types
of \cite[\S 3.1]{MR2202260}, corresponding to orbits of $PGL_4$.
We consider tensors up to isomorphism.

\begin{itemize}
\item One-regular algebras - centralizers of regular elements. There are $5$
types, their ranks are provided by Proposition \ref{centralizerprop}.
These are $\cO_{12},\cO_{11},\cO_{10'},\cO_{10''},\cO_{9}$ in \cite[\S
3.1]{MR2202260}.
\item Containing a $(3,1)$ Jordan type non-regular element. There are two types, giving rise to isomorphic
rank $6$ tensors. Set theoretically,   the intersection with the
Segre variety is a line and a point.
\[
T_{\cO_8'',\cO_8'}=
\begin{pmatrix}
c&0&a&0\\
0&c&b&0\\
0&0&c&0\\
0&0&0&d\\
\end{pmatrix}
\]
\item Containing a nilpotent element with Jordan block size $3$. There are three
types, all of rank $7$,
the first one representing a class to which the Coppersmith-Winograd tensor
$\tilde T_{2,CW }$ belongs, see \S\ref{cwtensors}. In the
second  case the intersection with the Segre set-theoretically is a line.
\[
T_{\cO_8}=\tilde T_{2,CW} =
\begin{pmatrix}
c&b&a&d\\
0&c&b&0\\
0&0&c&0\\
0&0&d&c\\

\end{pmatrix}
 ,
\ T_{\cO_7'',\cO_7'}=\begin{pmatrix}
c&b&a&0\\
0&c&b&0\\
0&0&c&0\\
0&0&d&c\\
\end{pmatrix}
 ,
\]
\item Four types, all of rank $7$, giving rise to three
different types of tensors. Set-theoretically the intersection with the Segre
is, in the
first case two lines intersecting in a point, in the second case a smooth
quadric, in the third   case a plane.
\[
T_{\cO_6}=
\begin{pmatrix}
c&0&a&b\\
0&c&0&d\\
0&0&c&0\\
0&0&0&c\\

\end{pmatrix}
 ,\
T_{\cO_7}=
\begin{pmatrix}
c&0&a&d\\
0&c&b&a\\
0&0&c&0\\
0&0&0&c\\

\end{pmatrix}
 ,\
T_{\cO_3',\cO_3''}=\begin{pmatrix}
c&a&b&d\\
0&c&0&0\\
0&0&c&0\\
0&0&0&c\\
\end{pmatrix}
  .
\]

\end{itemize}

\subsection{Case $m=5$}\label{brankfive} As remarked in \cite{MR2202260}, it is sufficient to consider nilpotent subspaces
as others are built out of them, so we restrict our attention to them.
Up to transpositions the following are the only maximal, nilpotent, $\tend$-abelian $5$-dimensional
subalgebras of the algebra of $5\times 5$ matrices.
   We prove that each of them
is  in  $Red(5)$. Notation is such that
$T_{N_{i,j}}$ corresponds to the nilpotent algebras $N_i,N_j$ of
\cite{MR2118458}, and we slightly abuse notation, identifying the tensor with its corresponding linear space.
\begin{align*}
&T_{N_{1,4}}=\begin{pmatrix}
a&0&0&0&0\\
b&a&0&0&0\\
c&0&a&0&0\\
d&0&0&a&0\\
e&0&0&0&a\\
\end{pmatrix}
 , \  \
T_{N_{6,8}}=\begin{pmatrix}
a&0&0&0&0\\
b&a&0&e&0\\
c&b&a&d&e\\
0&0&0&a&0\\
e&0&0&0&a\\
\end{pmatrix}
 ,
\ \
T_{N_{7,9}}=\begin{pmatrix}
a&0&0&0&0\\
b&a&0&d&0\\
c&b&a&e&d\\
0&0&0&a&0\\
0&0&0&b&a\\
\end{pmatrix}
 ,
\\
&
T_{N_{10,12}}=\begin{pmatrix}
a&0&0&0&0\\
b&a&0&0&0\\
c&b&a&0&0\\
d&0&0&a&0\\
e&0&0&0&a\\
\end{pmatrix}
 , \ \
T_{N_{11,13}}=\begin{pmatrix}
a&0&0&0&0\\
b&a&0&0&0\\
c&b&a&d&0\\
d&0&0&a&0\\
e&0&0&0&a\\
\end{pmatrix}
 , \ \
T_{N_{14}}= \tilde T_{3,CW} = \begin{pmatrix}
a&0&0&0&0\\
b&a&0&0&0\\
c&b&a&d&e\\
d&0&0&a&0\\
e&0&0&0&a\\
\end{pmatrix}
 ,\\
&
T_{N_{15}}=\begin{pmatrix}
a&0&0&0&0\\
b&a&0&0&0\\
c&b&a&0&0\\
d&c&b&a&0\\
e&0&0&0&a\\
\end{pmatrix}
 , \ \
T_{N_{16}}=\begin{pmatrix}
a&0&0&0&0\\
b&a&0&0&0\\
c&b&a&0&0\\
d&c&b&a&e\\
e&0&0&0&a\\
\end{pmatrix}
 , \ \
T_{N_{17}}=\begin{pmatrix}
a&0&0&0&0\\
b&a&0&0&0\\
c&b&a&0&0\\
d&c&b&a&e\\
0&0&0&0&a\\
\end{pmatrix}
 .
\end{align*}
$T_{N_{1,4}}$ is  obviously of border rank five.
For $T_{N_{6,8}}$,$T_{N_{15}}$,$T_{N_{16}}$, and $T_{N_{17}}$ (resp.~$T_{N_{7,9}}$) apply
Proposition \ref{prop:regular} to a pair of matrices represented by $b$ and $e$
(resp.~$b$,$d$).
    $T_{N_{10,12}}$
is the limit as $\ep\ra 0$ of the space spanned by the following five matrices:  
\[
\begin{pmatrix}
0&0&0&0&0\\
\epsilon&\epsilon^2&0&0&0\\
1&\epsilon&0&0&0\\
0&0&0&0&0\\
0&0&0&0&0\\
\end{pmatrix}
 ,\
\begin{pmatrix}
0&0&0&0&0\\
0&0&0&0&0\\
0&0&0&0&0\\
1&0&0&\epsilon^2&0\\
0&0&0&0&0\\
\end{pmatrix}
 ,\
\begin{pmatrix}
0&0&0&0&0\\
0&0&0&0&0\\
0&0&0&0&0\\
0&0&0&0&0\\
1&0&0&0&\epsilon^2\\
\end{pmatrix},
\begin{pmatrix}
0&0&0&0&0\\
0&0&0&0&0\\
1&0&0&0&0\\
0&0&0&0&0\\
0&0&0&0&0\\
\end{pmatrix}
 ,\
\begin{pmatrix}
\epsilon^2&-\epsilon^5&\epsilon^6&0&0\\
-\epsilon&\epsilon^4&-\epsilon^5&0&0\\
\epsilon^{-2}&-\epsilon&\epsilon^2&0&0\\
-1&\epsilon^3&-\epsilon^4&0&0\\
-1&\epsilon^3&-\epsilon^4&0&0\\
\end{pmatrix}
 .\]
$T_{N_{11,13}}$ is the limit of the space spanned by 
\[
\begin{pmatrix}
0&0&0&0&0\\
\epsilon&\epsilon^2&0&0&0\\
1&\epsilon&0&0&0\\
0&0&0&0&0\\
0&0&0&0&0\\
\end{pmatrix}
 ,
\begin{pmatrix}
0&0&0&0&0\\
0&0&0&0&0\\
\epsilon^{-3}&0&0&1&0\\
1&0&0&\epsilon^3&0\\
0&0&0&0&0\\
\end{pmatrix}
 ,
\begin{pmatrix}
0&0&0&0&0\\
0&0&0&0&0\\
0&0&0&0&0\\
0&0&0&0&0\\
1&0&0&-\epsilon^2&\epsilon^2\\
\end{pmatrix},
\begin{pmatrix}
0&0&0&0&0\\
0&0&0&0&0\\
1&0&0&0&0\\
0&0&0&0&0\\
0&0&0&0&0\\
\end{pmatrix}
 ,
\begin{pmatrix}
\epsilon^2&-\epsilon^5&\epsilon^6&-\epsilon^4&0\\
-\epsilon&\epsilon^4&-\epsilon^5&\epsilon^3&0\\
\epsilon^{-2}&-\epsilon&\epsilon^2&-1&0\\
-1&\epsilon^3&-\epsilon^4&\epsilon^2&0\\
-1&\epsilon^3&-\epsilon^4&\epsilon^2&0\\
\end{pmatrix}
 .\]
$T_{N_{14}}$ is isomorphic to the Coppersmith-Winograd tensor.
We determine the rank of each tensor.

The following conjecture was presented at the 2011 {\it Algebraic geometry with a view to applications} semester at the Mittag-Leffler institute:

\begin{conjecture} [J.~Rhodes] \cite[Conjecture 0]{MR3069955} 
The maximal rank of a tensor in $\BC^m\ot \BC^m\ot \BC^m$ of border rank $m$
is $2m-1$.
\end{conjecture}

This conjecture was known to be true classically for $m=1,2$ and verified in \cite{MR3069955} for $m=3$.
Above we verified it for $m=4$.
The following proposition shows that the conjecture is false for $m>4$:

\begin{proposition}\label{prop:brank5}
$$\bold R(T_{N_{1,4}})=\bold R(T_{N_{10,12}})=\bold R(T_{N_{11,13}})=\bold R(T_{N_{14}})=\bold R(T_{N_{15}})=\bold R(T_{N_{16}})=\bold R(T_{N_{1,4}})=9$$
and
$$\bold R(T_{N_{6,8}})=\bold R(T_{N_{7,9}})=10.$$
\end{proposition}
\begin{proof}
The fact that the rank of any tensor is at least $9$ follows by
the substitution method. To see
that $\bold R(T_{N_{6,8}})\geq 10$, first apply Proposition \ref{prop:slice} to the first and fourth row 
and to the third and fifth column. This shows that the rank is at least $4$ plus the rank of the tensor associated to 
\[\begin{pmatrix}
b&0&e\\
c+\alpha e&b+\beta e&d+\gamma e\\
e&0\\
\end{pmatrix},
\]
where $\alpha,\beta,\gamma$ are some constants. Now apply the proposition to the second column obtaining a tensor represented by 
\[\begin{pmatrix}
b&e\\
c+\alpha e+\delta b&d+\gamma e+\rho b\\
e&0\\
\end{pmatrix},
\]
where $\delta,\rho,\alpha,\gamma$ are (possibly new) constants. This space is 
 equal  to 
\[\begin{pmatrix}
b&e\\
c&d\\
e&0\\
\end{pmatrix}
\]
and it remains to show that it  corresponds to a tensor of  rank at least $5$. This follows by Proposition \ref{prop:slice} by first reducing $b,c,d$ and obtaining a rank $2$ matrix.

To prove that $R(T_{N_{7,9}})\geq 10$, apply Proposition \ref{prop:slice} and 
remove the second, third and fifth column and first and fourth row to obtain a tensor isomorphic to 
\[
\begin{pmatrix}
b&d\\
c&e\\
0&b\\
\end{pmatrix} , 
\]
and conclude as above.

The upper bounds for ranks of $T_{N_{1,4}}$, $T_{N_{10,12}}$, $T_{N_{15}}$ and $T_{N_{17}}$ follow  from Proposition \ref{centralizerprop}.

For $T_{N_{6,8}}$,  consider:
\begin{enumerate}
\item $5$ matrices, including the matrix corresponding to $c$, which follow from Proposition \ref{centralizerprop} for the upper left $3\times 3$ corner,
\item $2$ matrices for last $2$ diagonal entries,
\item $1$ matrix corresponding to $d$,
\item the $2$ matrices:
\[\begin{pmatrix}
0&0&0&0&0\\
0&0&0&1&0\\
0&0&0&0&0\\
0&0&0&0&0\\
0&0&0&0&0\\
\end{pmatrix},
\begin{pmatrix}
0&0&0&0&0\\
0&0&0&0&0\\
1&0&0&0&1\\
0&0&0&0&0\\
1&0&0&0&1\\
\end{pmatrix}.
\]
\end{enumerate}
 
For $T_{N_{11,13}}$ it is enough to notice that once a matrix for $c$ and the fourth diagonal entry are given, one can generate the matrix corresponding to $d$ using 
\[
\begin{pmatrix}
0&0&0&0&0\\
0&0&0&0&0\\
1&0&0&1&0\\
1&0&0&1&0\\
0&0&0&0&0\\
\end{pmatrix}.
\]
An analogous method shows $\bold R(T_{N_{16}})=9$. The tensor $T_{N_{14}}$ is a special case of Proposition \ref{prop:CWrank}.

For $T_{N_{7,9}}$, consider $3$ rank one matrices corresponding to entries of the first column, $2$ rank one matrices corresponding to the  third and fourth diagonal entry and one matrix corresponding to $e$. Apart from these six matrices we are left with the tensor represented by 
\[
\begin{pmatrix}
a&d&0\\
b&0&d\\
0&b&a\\
\end{pmatrix}.
\]
This tensor is isomorphic to the symmetric tensor given by 
the monomial $xyz$ which has  Waring rank $4$, see, e.g., \cite{MR2628829} (the upper bound dates
back at least to \cite{MR1573008}),  and thus tensor rank 
at most $4$.
\end{proof}

Apart from the nilpotent algebras just discussed there are two families of End-closed $5$-dimensional
subalgebras.
\begin{enumerate}
\item The subspace spanned by identity and any $4$-dimensional subspace of the $6$-dimensional algebra 
\[
\begin{pmatrix}
0&0&0&0&0\\
0&0&0&0&0\\
0&0&0&0&0\\
a&b&c&0&0\\
d&e&f&0&0\\
\end{pmatrix}
 .\]

In this case there are  normal forms: Tensors in
$A'\ot B'\ot C'=\BC^4\ot \BC^2\ot \BC^3 $ are classified
in   \cite{MR2996361}.
We have
$
T=a_1(b_1c_1+..+b_5c_5)
+ a_2b_1c_4+a_3b_1c_5
+ a_{ 4 }b_2c_4+a_{ 5 }b_2c_5
+ a_6b_3c_4+a_7b_3c_5= a_1(b_1c_1+..+b_5c_5)+T'
$
where $a_2,..a_7$ satisfy two linear relations.
If we make a change of basis in $c_4,c_5$,
say by a $2\times 2$ matrix $X$, then
as long as we change $b_4,b_5$ by $X\inv$
the first term does not  change.
  Similarly, if we make a change of basis in $b_1,b_2,b_3$,
by a matrix $Y$, then as long as we change
$c_1,c_2,c_3$ by $Y\inv$, the first term does not  change.
In our case we
may assume the tensors are $A$-concise. There are the following
cases (numbers as in \cite{MR2996361}),
 we abuse notation, writing $T$ for $T(A^*)$:  
 \begin{align*}
T_9 &=\begin{pmatrix}
x_1 & & & & \\
& x_1 & & & \\
& & x_1 & & \\
x_2&x_3 & & x_1 & \\
x_4&x_5 & & & x_1
\end{pmatrix}, \ \
T_{19}  =\begin{pmatrix}
x_1 & & & & \\
& x_1 & & & \\
& & x_1 & & \\
x_2&x_3 &x_5 & x_1 & \\
x_3&x_4 & & & x_1
\end{pmatrix} 
T_{20}  =\begin{pmatrix}
x_1 & & & & \\
& x_1 & & & \\
& & x_1 & & \\
x_2&x_4 &x_5 & x_1 & \\
x_3& & & & x_1
\end{pmatrix}, \\
T_{21}  &=\begin{pmatrix}
x_1 & & & & \\
& x_1 & & & \\
& & x_1 & & \\
x_2& x_4&x_5 & x_1 & \\
x_3&x_5 &  & & x_1
\end{pmatrix},
T_{22}  =\begin{pmatrix}
x_1 & & & & \\
& x_1 & & & \\
& & x_1 & & \\
x_2&x_4 & & x_1 & \\
x_3& &x_5 & & x_1
\end{pmatrix}, \ \
T_{23}  =\begin{pmatrix}
x_1 & & & & \\
& x_1 & & & \\
& & x_1 & & \\
x_2&x_3 &x_4 & x_1 & \\
x_3&x_4 &x_5 & & x_1
\end{pmatrix}.
\end{align*}
Now $\ur(T_9)=\ur(T_{20})=5$ because they are special cases of flag-algebra
tensors  - cf.~Proposition \ref{prop:Young}.

$T_{22} (A^*) $ is the limit of the space spanned by the following $5$ matrices:
\[\begin{pmatrix}
\epsilon^2 & &-\epsilon^3 & & \epsilon^4\\
& & & & \\
& & & & \\
& & & & \\
1& &-\epsilon & & \epsilon^2
\end{pmatrix},
\begin{pmatrix}
 & & & & \\
& & & & \\
& & & & \\
1&-\epsilon & & \epsilon^2 & \\
& & & &
\end{pmatrix},
\begin{pmatrix}
 & & & & \\
& \epsilon^2 & & & \\
& & & & \\
&\epsilon & & & \\
& & & &
\end{pmatrix},
\begin{pmatrix}
& & & & \\
& & & & \\
& & \epsilon^2 & & \\
& & & & \\
& &\epsilon & &
\end{pmatrix},
\begin{pmatrix}
 & & & & \\
&  & & & \\
& &  & & \\
1& & &  & \\
1& & & &
\end{pmatrix}.\]

$T_{23} (A^*) $ is the limit of the space spanned by  the following $5$ matrices:
\[\begin{pmatrix}
& & & & \\
 & & & & \\
& -2\epsilon^6& \epsilon^4& &\epsilon^8 \\
& & & & \\
 & -2\epsilon^2 &1 & &\epsilon^4
\end{pmatrix},
\begin{pmatrix}
 \epsilon^4& & &-\frac{1}{2}\epsilon^8 & \\
-\epsilon^3& & &\frac{1}{2}\epsilon^7 & \\
& & & & \\
-2& & & \epsilon^4 & \\
& & & &
\end{pmatrix},
\begin{pmatrix}
 & & & & \\
\epsilon^3&\epsilon^4  &2\epsilon^5 & & \\
& & & & \\
1&\epsilon & 2\epsilon^2& & \\
\epsilon& \epsilon^2& 2\epsilon^3& &
\end{pmatrix},
\begin{pmatrix}
& & & & \\
& & & & \\
& & & & \\
& &2\epsilon^2 & & \\
& &1+2\epsilon^3 & &
\end{pmatrix},
\begin{pmatrix}
 & & & & \\
&  & & & \\
& &  & & \\
1&-\epsilon & &  & \\
-\epsilon&\epsilon^{2} & & &
\end{pmatrix}.\]

$T_{21} (A^*) $ is the limit of the space spanned by  the following $5$ matrices:
\[\begin{pmatrix}
& & & & \\
&& & & \\
&\epsilon^3 &\epsilon^4 & & \\
&\epsilon & \epsilon^2& & \\
& \epsilon^2 &\epsilon^3 & &
\end{pmatrix},
\begin{pmatrix}
 & & & & \\
&-\epsilon^4 & & & \\
&\epsilon^3 & & & \\
& \epsilon& &  & \\
& & & &
\end{pmatrix},
\begin{pmatrix}
 \epsilon^4& &-\epsilon^6 &\epsilon^8& \\
&  & & & \\
& &  & & \\
1& & -\epsilon^2&\epsilon^4&\\
& & & &
\end{pmatrix},
\begin{pmatrix}
& & & & \\
& & & & \\
& & & & \\
& & & &\\
\epsilon&-\epsilon^2 &-\epsilon^3 & &\epsilon^4
\end{pmatrix},
\begin{pmatrix}
 & & & & \\
&  & & & \\
& &  & & \\
1& & & &\\
\epsilon& & &  &
\end{pmatrix}.\]

$T_{19} (A^*) $ is the limit of the space spanned by  the following $5$ matrices:
\[\begin{pmatrix}
\epsilon^2 & &-\epsilon^3 &\epsilon^4 & \\
 &  & & \\
& & & & \\
1& & -\epsilon &\epsilon^2 & \\
 & & & &
\end{pmatrix},
\begin{pmatrix}
 & & & & \\
& & & & \\
& & \epsilon^2& & \\
& &\epsilon & & \\
& & & &
\end{pmatrix},
\begin{pmatrix}
 & & & & \\
&\epsilon^2  & & &\epsilon^4 \\
& &  & &\ \\
& & & & \\
& 1& & &\epsilon^2
\end{pmatrix},
\begin{pmatrix}
& & & & \\
& & & & \\
& &  & & \\
1& 1& & & \\
1& 1& & &
\end{pmatrix},
\begin{pmatrix}
 & & & & \\
&  & & & \\
& &  & & \\
-1& 1& &  & \\
1& -1& & &
\end{pmatrix}.\]

\begin{proposition} $\bold R(T_{21})=10$
and all other tensors on this list have  $\bold R(T_j)=9$.
\end{proposition}
\begin{proof} This follows   by the substitution-method and considering
the ranks of $T'$ in \cite{MR2996361}.
\end{proof}

\item The subspace spanned by the identity and any $4$-dimensional subspace of the
 $5$-dimensional algebra 
 \be\label{fourdcase}
\begin{pmatrix}
0&0&0&0&0\\
a&0&0&0&0\\
b&a&0&e&0\\
0&0&0&0&0\\
d&0&0&c&0\\
\end{pmatrix}
 .
\ene
\end{enumerate}
 All the operations we will perform preserve the identity matrix.

First, by exchanging rows and  columns the algebra is 
$$\begin{pmatrix}
0&0&0&0&0\\
0&0&0&0&0\\
a&0&0&0&0\\
d&c&0&0&0\\
b&e&a&0&0\\
\end{pmatrix}
 .
 $$

Note that the tensor $T_{Leit,5}$ of  Proposition \ref{prop:Leitner} is in this
family (set $b=0$).

\begin{proposition}\label{fourdprop}  All End-closed  tensors in $\BC^5\ot \BC^5\ot \BC^5$
obtained from $4$-dimensional subspaces of \eqref{fourdcase} have
border rank  five.
\end{proposition}

To prove the proposition, we will use the following lemma:

\begin{lemma}\label{lem:brank5}
Each of the  tensors in $\BC^5\ot \BC^5\ot \BC^5$ corresponding to linear spaces spanned by the identity
and the following subspaces have border rank $5$:
 
$$S_1=\begin{pmatrix}
0&0&0&0&0\\
0&0&0&0&0\\
a&0&0&0&0\\
d&0&0&0&0\\
b&e&a&0&0\\
\end{pmatrix}
, S_2=
 \begin{pmatrix}
0&0&0&0&0\\
0&0&0&0&0\\
a&0&0&0&0\\
d&a&0&0&0\\
b&e&a&0&0\\
\end{pmatrix}
, S_3=\begin{pmatrix}
0&0&0&0&0\\
0&0&0&0&0\\
a&0&0&0&0\\
d&e&0&0&0\\
b&0&a&0&0\\
\end{pmatrix}
 , S_4= \begin{pmatrix}
0&0&0&0&0\\
0&0&0&0&0\\
a&0&0&0&0\\
d&e&0&0&0\\
b&d&a&0&0\\
\end{pmatrix} 
 .
$$

The tensors corresponding to $S_1,S_3,S_4$,   have rank $9$ and the tensor corresponding to $S_2$ has rank $10$.
\end{lemma}
\begin{proof}
We first prove the statement about the border rank. 
The span of $S_1$ and the identity  is the limit of the space spanned by 
\[\begin{pmatrix}
0&0&0&0&0\\
0&0&0&0&0\\
\epsilon&0&\frac{1}{2}\epsilon^2&0&0\\
0&0&0&0&0\\
2&0&\epsilon&0&0\\
\end{pmatrix}
 ,
 \begin{pmatrix}
\epsilon^2&0&-\frac{1}{2}\epsilon^3&0&\frac{1}{2}\epsilon^4\\
0&0&0&0&0\\
-\epsilon&0&\frac{1}{2}\epsilon^2&0&-\frac{1}{2}\epsilon^3\\
0&0&0&0&0\\
2&0&-\epsilon&0&\epsilon^2\\
\end{pmatrix}
 ,
\begin{pmatrix}
0&0&0&0&0\\
0&0&0&0&0\\
0&0&0&0&0\\
\epsilon&\frac{\epsilon^2}{4}&0&\epsilon^2&0\\
0&0&0&0&0\\
\end{pmatrix}
 ,
\]
\[
\begin{pmatrix}
0&0&0&0&0\\
0&\epsilon^2&0&0&0\\
0&0&0&0&0\\
0&0&0&0&0\\
0&\epsilon&0&0&0\\
\end{pmatrix},
\begin{pmatrix}
0&0&0&0&0\\
0&0&0&0&0\\
0&0&0&0&0\\
\epsilon&\frac{\epsilon^2}{4}&0&0&0\\
4&\epsilon&0&0&0\\
\end{pmatrix}.
\]
For $S_2$, consider
\begin{align*}X_1&= \begin{pmatrix}
\epsilon^4&0&\epsilon^6&0&\epsilon^8\\
0&0&0&0&0\\
\epsilon^2&0&\epsilon^4&0&\epsilon^6\\
0&0&0&0&0\\
1&0&\epsilon^2&0&\epsilon^4\\
\end{pmatrix},\
X_2
=\begin{pmatrix}
0&0&0&0&0\\
\alpha\beta\epsilon^3&\beta\epsilon^4&0&\alpha\beta^2\epsilon^6&0\\
0&0&0&0&0\\
\epsilon&\alpha^{-1}\epsilon^2&0&\beta\epsilon^4&0\\
\alpha&\epsilon&0&\alpha\beta\epsilon^3&0\\
\end{pmatrix},\
X_3 =\begin{pmatrix}
0&0&0&0&0\\
0&0&0&0&0\\
0&0&0&0&0\\
\epsilon&\frac{1-\alpha}{\alpha}\epsilon^2&0&0&0\\
1+\alpha&\epsilon&0&0&0\\
\end{pmatrix},
\\
X_4&=\begin{pmatrix}
0&0&0&0&0\\
0&0&0&0&0\\
0&0&0&0&0\\
0&0&0&0&0\\
0&\frac{\alpha}{\beta(\alpha-1)}\epsilon&\epsilon^2&\alpha\epsilon^3&0\\
\end{pmatrix},
X_5=\begin{pmatrix}
0&0&0&0&0\\
\alpha\epsilon^3&0&0&0&0\\
\epsilon^2&0&0&0&0\\
\frac{\alpha}{\beta(\alpha-1)}\epsilon&0&0&0&0\\
0&0&0&0&0\\
\end{pmatrix},
\end{align*}
where in order that $X_3$ has rank one, take $\a$ to be a solution
of the equation  $\a^2+\a-1=0$. We determine $\b$ later.

 First note that the limits as $\ep$ goes to zero of
$$\frac{1}{\epsilon}X_5, X_1,\frac{1}{\epsilon}X_4$$
give   matrices corresponding respectively to $d,b,e$.

Consider $X_1+X_2-X_3$. The constant terms and the terms of order $\epsilon$ add to zero.
The $(4,2)$ entry equals 
$$(\frac{1}{\alpha}-\frac{1-\alpha}{\alpha})\ep^2=\ep^2.$$
Hence, the limit gives the matrix corresponding to $a$.

It remains to prove that the identity matrix belongs to the limit. For this we consider
$$X_1+\frac{1}{\beta}(X_2-\frac{1}{1-\alpha}X_3)-X_4-X_5.$$
As   $\epsilon^4$ is on the diagonal it remains to prove that the lower order terms all add to zero.
For the constant term
$$1+\frac{\alpha}{\beta}-\frac{1+\alpha}{\beta(1-\alpha)}=
\frac{\beta(1-\alpha)+\alpha-\alpha^2-1-\alpha}{\beta(1-\alpha)} , 
$$
the numerator equals
$\beta(1-\alpha)-\alpha^2-1 $, so we take $\b=\frac{\a^2+1}{1-\a}$ to make it zero.
To see that the terms proportional to $\epsilon$ cancel, observe  that 
$$\frac{1}{\beta}(1-\frac{1}{1-\alpha})=-\frac{\alpha}{\beta(\alpha-1)}.$$

All three of the terms proportional to $\epsilon^2$ cancel, the only nontrivial being 
$$\frac{1}{\alpha}-\frac{1}{1-\alpha}\frac{1-\alpha}{\alpha}.$$
The term proportional to $\epsilon^3$ also cancels out.

The span of $S_3$ and the identity is the limit of the space spanned by 
\[
\begin{pmatrix}
0&0&0&0&0\\
0&0&0&0&0\\
-\epsilon&0&\frac{1}{2}\epsilon^2&0&0\\
0&0&0&0&0\\
2&0&-\epsilon&0&0\\
\end{pmatrix}
 ,
\begin{pmatrix}
\epsilon^2&0&\frac{1}{2}\epsilon^3&0&\frac{1}{2}\epsilon^4\\
0&0&0&0&0\\
\epsilon&0&\frac{1}{2}\epsilon^2&0&\frac{1}{2}\epsilon^3\\
0&0&0&0&0\\
2&0&\epsilon&0&\epsilon^2\\
\end{pmatrix}
,
\begin{pmatrix}
0&0&0&0&0\\
0&0&0&0&0\\
0&0&0&0&0\\
1&-{\epsilon}&0&\epsilon^2&0\\
0&0&0&0&0\\
\end{pmatrix}
 ,
\]
$$
\begin{pmatrix}
0&0&0&0&0\\
0&\epsilon^2&0&0&0\\
0&0&0&0&0\\
0&\epsilon&0&0&0\\
0&0&0&0&0\\
\end{pmatrix},
\begin{pmatrix}
0&0&0&0&0\\
0&0&0&0&0\\
0&0&0&0&0\\
1&0&0&0&0\\
4&0&0&0&0\\
\end{pmatrix}.
$$

The span of $S_4$ and the identity is the limit of the space spanned by 
\[
\begin{pmatrix}
0&0&0&0&0\\
0&0&0&0&0\\
0&0&0&0&0\\
0&0&0&0&0\\
1&0&0&0&0\\
\end{pmatrix}
 ,
\begin{pmatrix}
\epsilon^6&\epsilon^8&0&0&\epsilon^{12}\\
0&0&0&0&0\\
0&0&0&0&0\\
\epsilon^2&\epsilon^4&0&0&\epsilon^8\\
1&\epsilon^2&0&0&\epsilon^6\\
\end{pmatrix},
\begin{pmatrix}
0&0&0&0&0\\
0&0&0&0&0\\
\epsilon^3&0&\epsilon^6&0&0\\
0&0&0&0&0\\
1&0&\epsilon^3&0&0\\
\end{pmatrix}
 ,
\]
\[
\begin{pmatrix}
0&0&0&0&0\\
0&\epsilon^6&0&-\epsilon^{8}&0\\
0&0&0&0&0\\
0&-\epsilon^4&0&\epsilon^6&0\\
0&0&0&0&0\\
\end{pmatrix}
 ,
\begin{pmatrix}
0&0&0&0&0\\
0&0&0&0&0\\
-\epsilon^3&\epsilon^8&\epsilon^9&0&0\\
-\epsilon^2&\epsilon^7&\epsilon^8&0&0\\
\epsilon^{-3}&-\epsilon^2&-\epsilon^3&0&0\\
\end{pmatrix}
 . 
\]

 The   lower bounds for rank follow by substitution method. For rank upper bounds, consider the seven rank one matrices:
\begin{enumerate}
\item $4$ matrices corresponding to first two and last two entries of the diagonal,
\item $1$ matrix corresponding to $b$,
\item the $2$ matrices:
\[
\begin{pmatrix}
0&0&0&0&0\\
0&0&0&0&0\\
-1&0&1&0&0\\
0&0&0&0&0\\
1&0&-1&0&0\\
\end{pmatrix}
 ,
\begin{pmatrix}
0&0&0&0&0\\
0&0&0&0&0\\
1&0&1&0&0\\
0&0&0&0&0\\
1&0&1&0&0\\
\end{pmatrix}
 . 
\]
\end{enumerate}
In all four cases it is easy to find the remaining rank $1$ matrices. 
\end{proof}

\begin{proof}[Proof of Proposition \ref{fourdprop}]
 If the subspace is given by $a=0$ then we conclude by Proposition \ref{prop:Young}.
Otherwise there exists a matrix $M_1$ in the algebra with the
entries corresponding to $a$ nonzero.
We may assume that the $3$ other generators of the algebra $M_2,M_3,M_4$ have
entries corresponding to $a$ equal to zero.
Further, we may assume that $M_2$ has only one entry nonzero, corresponding to
$b$, as otherwise, by considering $M_1^2$ the algebra would not be $\tend$-closed. Hence we may assume that $M_3$ and $M_4$ have only nonzero
entries on $d$, $c$ and $e$.

Let $\tilde S$ denote the $3$-dimensional vector space corresponding
to $d$, $c$ and $e$, and let $S\subset \tilde S$ be the two-dimensional subspace
spanned by $M_3$ and $M_4$.

Case 1)   $S=\{ M\in \tilde S\mid c=0\}$. We may assume
that $M_3$ corresponds to $e$, $M_4$ corresponds to  $d$ and the algebra is given
by 
\[\begin{pmatrix}
0&0&0&0&0\\
0&0&0&0&0\\
a&0&0&0&0\\
d&\lambda a&0&0&0\\
b&e&a&0&0\\
\end{pmatrix}
 ,\]
for some constant $\lambda$. If $\lambda=0$ we are in case $S_1$. Otherwise,
by multiplying second column by $1/\lambda$ and second row by $\lambda$ we may
assume that $\lambda=1$ and are in case $S_2$.

Case 2)   $S=\{ M\in \tilde S\mid e=0\}$. Subtract 
any multiple
of the third column from the second column,  and add 
the same
multiplicity of the second row to the third row to reduce to   case $S_3$.

Case 3)   $S=\{ M\in \tilde S\mid d=0\}$. This is analogous to
Case 2).

Case 4) As we are not in case 2) or   3) we may assume there are constants $\l\in \BC$, and $\d\neq 0$ such that
\[ S=\begin{pmatrix}
0&0&0&0&0\\
0&0&0&0&0\\
0&0&0&0&0\\
d&\lambda d+\delta e&0&0&0\\
0&e&0&0&0\\
\end{pmatrix}
 . \]
  Then we may assume
that $M_1$ represents the space 
\[\begin{pmatrix}
0&0&0&0&0\\
0&0&0&0&0\\
a&0&0&0&0\\
0&0&0&0&0\\
0&\rho a&a&0&0\\
\end{pmatrix}
 .\]
Subtract $\rho$ times the third column from the second column, and
add $\rho$ times the second row to third row, reducing to $\rho=0$. At this
point we have the algebra 
\[\begin{pmatrix}
0&0&0&0&0\\
0&0&0&0&0\\
a&0&0&0&0\\
d&\lambda d+\delta e&0&0&0\\
b&e&a&0&0\\
\end{pmatrix}
 .\]
Subtract  $1/\delta$ times the fourth row from the fifth row and add 
$1/\delta$ times the fifth column to the fourth column to obtain a subspace
 isomorphic to 
\[\begin{pmatrix}
0&0&0&0&0\\
0&0&0&0&0\\
a&0&0&0&0\\
d&e&0&0&0\\
b&-\frac{\lambda}{\delta}d&a&0&0\\
\end{pmatrix}
 .\]
If $\lambda=0$ we are in Case 2). If $\lambda\neq 0$ we multiply the second
column by $-\frac{\delta}{\lambda}$ and the second row by
$-\frac{\lambda}{\delta}$ to reduce to case $S_4$.
\end{proof}

\subsection{ Examples with large gaps between rank and border
rank} \label{lgex}
 First   consider
$$T_{gap}=a_1\ot (b_1\ot c_1 + \cdots + b_6\ot c_6)+ a_2\ot b_6\ot c_1 + a_3\ot
(b_5\ot c_1+b_6\ot c_3)+a_4\ot b_5\ot c_2+ a_5\ot b_4\ot c_3 +
a_6\ot (b_4\ot c_2+b_5\ot c_3),$$
i.e.,
$$
T_{gap}(A^*)=
\begin{pmatrix}
x_1&&&&&\\
&x_1&&&&\\
&&x_1&&&\\
&x_6&x_5&x_1&&\\
x_3&x_4&x_6&&x_1&\\
x_2&&x_3&&&x_1\\
\end{pmatrix}.
$$
Note $T_{gap}$ is neither $1_B$ nor $1_C$-generic.  In the following proposition we show that $T_{gap}$ is also a counterexample to \cite[Conjecture 0]{MR3069955}. 
\begin{proposition}
 $\ur(T_{gap})=6$ and $\bold R(T_{gap})=12$.
 \end{proposition}
\begin{proof}
To prove that $\bold R(T_{gap})= 12$, by the substitution method, it suffices to prove that
$\bold
R(T')=6$, where $T'$ corresponds to the subspace 
$$
\begin{pmatrix}
&x_6&x_5\\
x_3&x_4&x_6\\
x_2&&x_3\\
\end{pmatrix}.
$$
This space is contained in the space spanned by the following $6$ rank
one matrices:
\begin{enumerate}
\item 3 matrices corresponding to  $x_2,x_4,x_5$,
\item 2 matrices corresponding to $x_6$,
\item the matrix
$$
\begin{pmatrix}
&&\\
1&&1\\
1&&1\\
\end{pmatrix}.
$$
\end{enumerate}
Were $T'$  of rank $5$,  then the space would be equal to the space
generated by five rank one matrices. However, each rank one matrix in this space
has  $x_3=0$. This finishes the proof that $\bold R(T_{{gap}})=12$.
It remains to prove $\ur(T_{gap})=6$. As $T_{gap}$ is $A$-concise we only need
to prove
$\ur(T_{gap})\leq 6$. Consider the following $6$ rank one matrices:
$$
\begin{pmatrix}
&&&&&\\
&&&&&\\
&&&&&\\
&&&&&\\
&&&&&\\
1&&&&&\\
\end{pmatrix},
\begin{pmatrix}
&&&&&\\
&&&&&\\
&&&&&\\
&&1&&&\\
&&&&&\\
&&&&&\\
\end{pmatrix},
\begin{pmatrix}
&&&&&\\
&&&&&\\
&&&&&\\
&&&&&\\
-\epsilon^4&-1&&&&\\
&&&&&\\
\end{pmatrix},
$$
$$
\begin{pmatrix}
\epsilon^5&&\epsilon^9&&&\epsilon^{10}\\
&&&&&\\
&&&&&\\
&&&&&\\
\epsilon^4&&\epsilon^8&&&\epsilon^9\\
1&&\epsilon^4&&&\epsilon^5\\
\end{pmatrix},
\begin{pmatrix}
&&&&&\\
&\epsilon^5&\epsilon^8&&\epsilon^{10}&\\
&&&&&\\
&\epsilon^3&\epsilon^6&&\epsilon^8&\\
&1&\epsilon^3&&\epsilon^5&\\
&&&&&\\
\end{pmatrix},
\begin{pmatrix}
&&&&&\\
&&&&&\\
&\epsilon^8&\epsilon^5&-\epsilon^{10}&&\\
&-\epsilon^3&-1&\epsilon^5&&\\
&-\epsilon^6&-\epsilon^3&\epsilon^8&&\\
&-\epsilon^7&-\epsilon^4&\epsilon^9&&\\
\end{pmatrix}.
$$
\end{proof}

So far all the tensors we considered had rank less than  or equal to twice the border rank.
By \cite{MR3368091}
the maximal rank of a tensor is at most twice the maximal border rank, and all previously known
examples of tensors had rank at most twice the border rank. The following example
goes beyond this ratio.

For $T\in A\ot B\ot C$ and $T'\in A'\ot B'\ot C'$, consider $A\ot A'$  as a single vector space and similarly for $B\ot B'$ and $C\ot C'$. Define $T\ot T'\in (A\ot A')\ot (B\ot B')\ot (C\ot C')$. In what follows, we will use a tensor $T'\in \BC^2\ot \BC^2\ot \BC^2$  to
produce three copies of   $T_{gap}$.

\begin{proposition}\label{biggapprop}
The tensor
$$T_{ biggap }:=T_{gap}\otimes (e_1\otimes f_1\otimes g_1 +e_2\otimes(f_2\otimes g_1+f_1\otimes g_2))
\in(\BC^{12})^{\otimes 3}$$ has border rank $12$ and rank at least $25$.
\end{proposition}
\begin{proof}
$T_{biggap}$ has border rank $12$ as it is the tensor product of concise tensors of border rank $6$ and $2$.
In terms of matrices,
$$
T_{biggap}(A^*\ot \BC^{2*})=
\begin{pmatrix} T_{gap}(A^*) & T_{gap}(\tilde A^*)\\  T_{gap}(\tilde A^*)&
\end{pmatrix}
$$
where $\tilde A^*$ is another copy of $A^*$.
Denote the vectors in $T(\tilde A^*)$ with primes. Eliminate $x_1,x_1'$ by the substitution method so that
$T_{biggap}$ has border rank at least $9$ plus the rank of
 the tensor $\tilde T\in \BC^6\ot\BC^6\ot\BC^{10}$ represented by 
\[
\begin{pmatrix}
0&x_6&x_5&0&x_6'&x_5'\\
x_3&x_4&x_6&x_3'&x_4'&x_6'\\
x_2&0&x_3&x_2'&0&x_3'\\
0&x_6'&x_5'&0&0&0 \\
x_3'&x_4'&x_6'&0&0&0 \\
x_2'&0&x_3'&0&0&0
\end{pmatrix}.
\]
 We  now prove $\bold R(\tilde T)\geq 13$.
 Write $\tilde T=x_2\ot M_2+\cdots + x_6'\ot M_{6}'$.
 Apply Proposition \ref{prop:slice}  first with
 $x_2$ to get a tensor
 $\tilde T^{(1)}=x_3\ot (M_3-\l_3M_2)+ \cdots + x_6'\ot (M_{6}'-\l_{6}'M_2)$
 with $\bold R(\tilde T^{(1)})\leq \bold R(\tilde T)-1$. Continue in this manner,
 eliminating all but $x_3'$ to get a tensor
 $\tilde T^{(9)}=x_3\ot (M_3 +c_2M_2+\cdots + c_{6}'M_6')\in \BC^1\ot \BC^6\ot \BC^6$ where
 the $c_j,c_j'$ are some constants.
  
Hence $\bold R(\tilde T)\geq 9+ \bold R(\tilde T^{(9)})$.  But $ \bold R(\tilde T^{(9)})$ is simply the (usual) rank
of the matrix 
\[\tilde M_3=
\begin{pmatrix}
0&c_6&c_5&0&c_6'&c_5'\\
c_3&c_4&c_6&1&c_4'&c_6'\\
c_2&0&c_3&c_2'&0&1\\
0&c_6'&c_5'&0&0&0 \\
1&c_4'&c_6'&0&0&0 \\
c_2'&0&1&0&0&0
\end{pmatrix}.
\]
It remains to show  that $\trank(\tilde M_3)\geq 4$. Suppose to
the contrary that all the $4\times 4$ minors of $\tilde M_3$ are zero. One of the minors equals $(1-c_2'c_6')^2$, 
hence we would need $c_2',c_6'\neq 0$. 
There is also a minor $(c_2'c_5')^2$, which would force $c_5'=0$. 
However, there is also the  minor $(c_4'c_5'-c_6'^2)^2$ which under these assumptions  cannot be zero.
We conclude $\bold R(T)\geq 12+9+4=25$. 
\end{proof}

By further tensoring $T_{gap}$ analogously as above, we obtain tensors with rank to border rank ratio converging at least to $ 13/6 $.

 The following tensor is a generalization of $S_2$ of Lemma \ref{lem:brank5}, which is the case $T_{biggap,5}$.
\begin{theorem}\label{biggergapthm}
Let $m=2k+1$. Let 
$\tilde T\in \BC^{m-1}\ot\BC^{k+1}\ot\BC^{k+1}$ be the tensor represented by 
\[
\begin{pmatrix}
x_0&0&0&\dots &0&0\\
x_1&x_0&0&\dots&0&0\\
x_2&0&x_0&\vdots&0&0\\
\vdots&&&&&\\
x_{k-1}&0&0&\dots&x_0&0 \\
x_k&x_{k+1}&x_{k+2}&\dots &x_{2k-1}&x_0
\end{pmatrix}.
\]
Let $T_{biggap,m}\in \BC^m\ot \BC^m\ot \BC^m$ be the tensor such that $T_{biggap,m}(A^*)$ has $\tilde T(\BC^{m-1*})$ in its lower
left $(k+1)\times (k+1)$ corner and $x_{2k}$ along the diagonal.
Then $m\leq \ur(T_{biggap,m})\leq m+1$ and $\bold R(T_{biggap,m})=\frac 52 (m-1)$. In particular
the ratio of rank to border rank tends to $\frac 52$ as $m\ra \infty$.
\end{theorem}
\begin{proof}
It is straightforward that $\bold R(\tilde T)=3k$. 
  By  the substitution method $\bold R(T_{biggap,m})\geq 2k+\bold R(\tilde T)=5k$ and 
in fact equality holds.  

To estimate the border rank, fix bases $\{ a_j\}$ of $A$, $\{ b_j\}$ of $B$ and $\{ c_j\}$ of $C$ so that
  $\tilde T$   belongs to the subspace $\tspan\{ a_1,\dots,a_{m-1}\}\ot 
  \tspan\{b_{k+1},\dots,b_m\}\ot\tspan\{e_1,\dots,e_{k+1}\}$.
Consider the following $m$ rank $1$ elements of $B\ot C$:
\begin{enumerate}
\item $(2b_m+\epsilon b_{k+1}+\epsilon^2b_1)\ot (c_1+\frac 12  \epsilon c_{k+1}+
\frac 12 \epsilon^2 c_m),$
\item $(2b_m - \epsilon b_{k+1}+\epsilon^2b_1)\ot (c_1-\frac 12  \epsilon c_{k+1}+
\frac 12 \epsilon^2 c_m),$
\item $(b_{k+1}+\dots+b_{m-1}+2b_m)\ot c_1,$
\item $b_m\ot(2c_1+c_2+\dots+c_k),$
\item $b_i\ot (c_1+\epsilon c_{i-k+1}+\epsilon^2c_{i})$ for $i=k+2,\dots,m-1,$
\item $(b_m - \epsilon b_{k+i}+\epsilon^2 b_i)\ot c_{i}$ for $i=2,\dots,k.$
\end{enumerate}
The rank one elements of $\tilde T(\BC^{m-1*})$ are obtained from 5) and 6). 
The diagonal of $\tilde T(\BC^{m-1*})$ is obtained by adding all
elements of 5) with  1) and subtracting 3). 
The identity matrix is obtained by adding all elements of  5) and 6) with  1) and 2) and subtracting 3) and 4).  
\end{proof}

\begin{remark}After we posted a preprint of this article on the arXiv, Jeroen Zuiddam
shared with us his forthcoming article (now \cite{2015arXiv150405597Z})  presenting
an example of a sequence of tensors with rank to border rank ratio approaching three.
\end{remark}

 \begin{question} Is the ratio of rank to border rank unbounded? Can one find explicit tensors with ratio $3$ or larger?
 \end{question}
 
In this context we recall the following problem, which is a variant of our question
in the situation of minimal border rank:
\begin{problem} \cite[Open problem 4.1]{blaser2014explicit}
  Is there an explicit family of tensors $T_m\in \BC^m\ot\BC^m\ot \BC^m$
   with $\bold R(T_m)\geq(3+\ep)m$ for some $\ep>0$?
  Can we even achieve this for tensors corresponding to the multiplication in an algebra, i.e., is there an explicit family of algebras $\cA_m$  with $\bold R(\cA_m)\geq  (3+\ep )\tdim\cA_m$ for some $\ep > 0$?   
\end{problem}
   
 For tensors $T\in A_1\otc A_n$,   there are tensors of rank $n-1$ of border rank two.

\subsection{There are parameters worth of non-isomorphic  $1$-generic border rank
$m$ tensors in $\BC^m\ot \BC^m\ot \BC^m$}

Let $\t\in Mat_{p\times n}$. Set $m=n+p+1$.  Define
$$
T_{Leit,\t}:=
a_1\ot (b_1\ot c_1+\cdots \ot b_{p+n+1}\ot c_{p+n+1})
+\sum_{j=1}^p a_{1+j}\ot b_j\ot ( \sum_{s=1}^n \t_{j,s}c_{p+1 +s})
+\sum_{s=1}^na_{p+1+s}b_{p+1}c_{p+1+s}
$$
Leitner \cite{MR3503693} shows that $\ur(T_{{Leit},\t})=m$ and that the family gives
non-isomorphic tensors for $p\geq 4$, $n\geq 2$, i.e., $m\geq 7$.

\begin{remark}  Leitner only shows the border rank condition under certain genericity
hypotheses
on $\t$, but from the border rank perspective they are  unnecessary by
taking limits.
(Border rank is semi-continuous.)
\end{remark}

In particular, when $p=4, n=2$   Leitner  shows that there is a one-parameter
family of non-isomorphic subgroups of $SL_n\BR$ that are limits under conjugacy of the 
torus. The same argument shows  that there is a corresponding one-parameter family of
non-isomorphic tensors.

\section{Coppersmith-Winograd value}\label{cwvaluesect}
As mentioned in the introduction, a   motivation for this article is the
study of upper bounds for the exponent of matrix
multiplication. For our purposes, the {\it exponent}  $\o$ of matrix multiplication, which governs the  complexity
of the matrix multiplication tensor $\Mn\in \BC^{n^2}\ot \BC^{n^2}\ot \BC^{n^2}$, may be defined as
$$\o:=\tinf\{ \t\in \BR\mid \ur(\Mn)=O(\nnn^{\t})\}.
$$
Na\"ively one has $\o\leq 3$ and it is generally conjectured by computer scientists that $\o=2$.

The \lq\lq proper\rq\rq\ way to determine the exponent of matrix
multiplication would be to determine the
border rank of the matrix multiplication tensor. Unfortunately, this appears to
be beyond our current capabilities.
Thanks to  a considerable amount of work, most notably
\cite{MR623057,MR882307,copwin135}, one  can prove upper bounds
for matrix multiplication by considering other tensors.

First,    Sch\"onhage's asymptotic sum inequality \cite{MR623057} states that for all $\lll_i,\mmm_i,\nnn_i$, with  $1\leq i\leq s$,
$$
\sum_{i=1}^s(\mmm_i\nnn_i\lll_i)^{\frac{\o}3}\leq \ur(\bigoplus_{i=1}^s M_{\langle \mmm_i,\nnn_i,\lll_i\rangle}).
$$
Then Strassen \cite{MR882307} pointed out that it would be sufficient to find upper bounds on the border rank of a tensor
that degenerated into a disjoint sum of matrix multiplication tensors. This was exploited most successfully
by Coppersmith and Winograd \cite{copwin135}, who attained their success with a tensor $\tilde T_{CW}$. The purpose of
this section is to isolate geometric aspects of this tensor in the hope of finding other tensors
that would enable further upper bounds on the exponent.

In practice, only
tensors of minimal, or near minimal border rank
have been used to prove upper bounds on the exponent. Call a tensor $T$ that gives a \lq\lq good\rq\rq\ upper bound for
the exponent via the methods of
\cite{MR882307,copwin135}, of {\it high Coppersmith-Winograd value} or {\it high
CW-value} for short.  More precisely, $T$ has high Coppersmith-Winograd value
if the quantity $Val_{\rho}(T^{\ot k})$
as defined in \cite[p. 8]{DBLP:journals/corr/AmbainisFG14} is large for some $k$.  We briefly review tensors that have been utilized.
  Our study is incomplete because the CW-value of
a tensor also
depends on its presentation, and in different
bases a tensor can have quite different CW-values. Moreover, even determining
the value
in a given presentation still involves some \lq\lq art\rq\rq\ in the choice of a
good decomposition, choosing the correct tensor power, estimating the value and probability of each block \cite{williams}.

\subsection{Sch\"onhage's tensors}
Sch\"onhage's tensors are  $T_{Sch}=M_{\langle N,1,1\rangle}\ot M_{\langle 1,m,n\rangle}$
where
$N=(m-1)(n-1)$.
Here $\ur(T_{Sch})=N+1$ while $\bold R(T_{Sch})=N+mn=2\ur(T_{Sch})-(m+n-1)$.
It gives $\o<2.55$. There is nothing to gain by taking tensor
powers here because the two matrix multiplications are already disjoint.

\subsection{The Coppersmith-Winograd tensors}\label{cwtensors}
 Coppersmith and Winograd define two tensors:
\be\label{Tcw}
T_{q,CW}:= \sum_{j=1}^q a_0\ot b_j\ot c_j+ a_j\ot b_0\ot c_j+ a_j\ot b_j\ot c_0
\in \BC^{q+1}\ot \BC^{q+1}\ot \BC^{q+1}
\ene
and
\be\label{tildeTcw}
\tilde T_{q,CW}:= \sum_{j=1}^q (a_0\ot b_j\ot c_j+ a_j\ot b_0\ot c_j+ a_j\ot
b_j\ot c_0 )
+a_0\ot b_0\ot c_{q+1}+ a_0\ot b_{q+1}\ot c_{0}+ a_{q+1}\ot b_0\ot c_0
\in \BC^{q+2}\ot \BC^{q+2}\ot \BC^{q+2}
\ene
 both of which have    border rank $q+2$.

In terms of matrices,
$$
T_{q,CW}(C^*)=\begin{pmatrix} 0&x_1&\cdots & &x_q\\
x_1& x_0 & 0  &\cdots & \\
x_2& 0  & x_0 & & \\
\vdots &\vdots & & \ddots & \\
x_q& 0 & \cdots & 0 &x_0
\end{pmatrix}
$$
and
$$
\tilde T_{q,CW}(C^*)= \begin{pmatrix} x_{q+1}&x_1&\cdots & &x_q&x_0 \\
x_1& x_0 & 0  &\cdots& & 0\\
x_2& 0  & x_0 & & & \\
\vdots &\vdots &  &\ddots & & \\
x_q& 0 & \cdots & 0& x_0  & \\
x_0& 0 & \cdots & & 0 & 0
\end{pmatrix} .
$$
Permuting bases, we may also write

$$
\tilde T_{q,CW}(C^*)= \begin{pmatrix} x_{0}& &  & & &   \\
x_1& x_0 & 0  &\cdots& & 0\\
x_2& 0  & x_0 & & & \\
\vdots &\vdots &  &\ddots & & \\
x_q& 0 & \cdots & 0& x_0  & \\
x_{q+1}& x_1 & \cdots & & x_q & x_0
\end{pmatrix} .
$$

\begin{proposition}\label{prop:CWrank}
$\bold R(T_{q,CW})=2q+1$, $\bold R(\tilde T_{q,CW})=2q+3.$
\end{proposition}
\begin{proof}
We first prove the lower bound for $T_{q,CW}$. Apply Proposition
\ref{prop:slice} to show that the rank of the tensor is at least $2q-2$ plus the
rank of 
$$
\begin{pmatrix}
0&x_0\\
x_0&x_1
\end{pmatrix},
$$
which has rank $3$. An analogous estimate provides the lower bound for  $\bold
R(\tilde T_{q,CW})$.
To show that $\bold R(T_{q,CW})\leq 2q+1$ consider the following rank $1$
matrices, whose span contains $T(A^*)$:

1) $q+1$ matrices with all entries equal to $0$ apart from one entry on the
diagonal equal to $1$,

2) $q$ matrices indexed by $1\leq j\leq q$, with all entries equal to zero apart
from the four entries  $(0,0), (0,j),(j,0),(j,j)$ equal to $1$.

For the tensor $\tilde T_{CW}$ we consider the same matrices, however both
groups
have one more element.
\end{proof}

Coppersmith and Winograd used $\tilde T_{CW}$ to show $\o<2.3755$. In subsequent
work Stothers \cite{stothers}, resp.  V. Williams \cite{williams},
resp. LeGall \cite{legall} used   $\tilde T_{CW}^{\ot 4}$
resp.  $\tilde T_{CW}^{\ot 8}$, resp. $\tilde T_{CW}^{\ot 16}$ and $\tilde T_{CW}^{\ot 32}$
leading to the current \lq\lq world record\rq\rq \ 
$\o<2.3728639$.

 Ambainis,  Filmus and LeGall \cite{DBLP:journals/corr/AmbainisFG14} showed that
 taking higher powers of $\tilde T_{CW}$ when $q\geq 5$
 cannot prove $\o<2.30$ by this method alone.
 Their suggestion that one should look for new tensors to prove further upper bounds 
   was one  motivation for this paper.

\subsection{Strassen's tensor} Strassen uses the following concise tensor to show
$\o<2.48$:
\be\label{strassupten}
T_{Str,q}=\sum_{j=1}^q a_0\ot b_j\ot c_j+ a_j\ot b_0\ot c_j\in \BC^{q+1}\ot
\BC^{q+1}\ot
\BC^q
\ene
which has border rank $q+1$, as the  the $q$ vectors $[a_0\ot b_j\ot c_j+a_j\ot b_0\ot c_j]$, $1\leq j\leq q$,  
are tangent vectors to $q$
points $[a_0\ot b_0\ot c_1]\hd [a_0\ot b_0\ot c_q]$ that lie on   the
 $\pp{q-1}=\BP \{ a_0\ot b_0\ot \langle c_1\hd c_q\rangle \}\subset Seg(\BP A\times \BP B\times \BP C)$. 
 Any linear combination of $q+1$ tangent vectors based at $q+1$ linearly dependent points, with any
 size $q$ subset independent, of
 any variety,  has border rank at most $q+1$, see \cite[\S 10.1]{MR2628829}. Here we just take one of the
 vectors to be zero.  Note that $T_{Str,q}$ is a specialization of $T_{CW,q}$  obtained by setting
$c_0=0$.  By the substitution method the rank of the tensor equals $2q$. 

The corresponding linear spaces are:
$$
T_{Str,q}(C^*)=\begin{pmatrix} 0&x_1&\cdots &x_q\\ x_1& 0 & \cdots  & 0\\
x_2& &  & \\
\vdots &\vdots &  &\vdots  \\
x_q& 0 & \cdots &0
\end{pmatrix},
$$
and
$$
T_{Str,q}(A^*)=
\begin{pmatrix} x_1& x_2 &\cdots & x_q\\
x_0 & 0 & \cdots &0\\
0& x_0 & 0 &\vdots  \\
\vdots & \ddots & & \\
0&\cdots  &0 & x_0\end{pmatrix}.
$$

Actually Strassen uses the tensor product of this tensor with its images under
$\BZ_3$
acting on the three factors:
  $\tilde T:=T\ot T'\ot T''$ where $T',T''$ are cyclic permutations of
$T=T_{Str,q}$.
Thus $\tilde T \in \BC^{q(q+1)^2}\ot \BC^{q(q+1)^2}\ot \BC^{q(q+1)^2}$
and $\ur(\tilde T)\leq (q+1)^3$.

\subsection{Extremal tensors}\label{flagexsect}

Let $A,B,C=\BC^m$. 
There are normal forms for germs of curves in $Seg(\BP A\times \BP B\times \BP C)$
up to order $m-1$, namely
$$
T_t=(a_1+ta_2+\cdots + t^{m-1}a_m+O(t^{m }))\ot 
(b_1+tb_2+\cdots + t^{m-1}b_m+O(t^{m  }))\ot (c_1+tc_2+\cdots + t^{m-1}c_m+O(t^{m }))
$$
and if the $a_j$, $b_j$, $c_j$ are each linearly independent sets of vectors,
we will call the curve {\it general} to order $m-1$.

\begin{proposition}
Let $T\in A\ot B\ot C=\BC^m\ot \BC^m\ot \BC^m$. Let $T_0^{(m-1)}(A^*):=\frac{d^{m-1}}{(dt)^{m-1}}|_{t=0}T_t(A^*)$.
If  
$$T(A^*)=
  T_0^{(m-1)}(A^*) ,
$$
with
$T_t$ a curve that is general to order $m$, then $T(A^*)$ is the centralizer of
a regular nilpotent element.
\end{proposition}
\begin{proof}
Note that $T_0^{(q)}=q!\sum_{i+j+k=q-3}a_i\ot b_j\ot c_k$, i.e.,
$$
T_0^{(q)}(A^*)=
\begin{pmatrix}
x_{q-2} & x_{q-3} & \cdots & \cdots  & x_1 & 0 &\cdots \\
x_{q-3} & x_{q-4} & \cdots & x_1 & 0 & \cdots &\cdots \\
  \vdots & &   &   &   &   & \\
x_1     & 0 & \cdots  &   &   &   & \\
0     & 0 & \cdots  &   &   &   & \\
\vdots       & \vdots   &    &   &   &   & \\
   0     & 0 & \cdots  &   &   &   &
\end{pmatrix}
$$
in particular,   each space contains the   previous ones, and the last equals 
$$
\begin{pmatrix} x_m& x_{m-1} & \cdots &  & x_1\\
x_{m-1} &x_{m-2}& \cdots & x_{1} & 0 \\
\vdots &\vdots & \udots& & \\
\vdots &x_1 & & & \\
x_1 & 0 & & &
\end{pmatrix}
$$
which is isomorphic to the centralizer of a regular nilpotent element. \end{proof}

This  provides another, explicit proof that the centralizer of a regular nilpotent element belongs
to the closure of diagonalizable algebras. 

\begin{proposition}\label{cwflagbb} 
 Let $T\in A\ot B\ot C=\BC^m\ot \BC^m\ot
\BC^m$ be   of border rank $m >2 $.
Assume $\BP T(A^*)\cap Seg(\BP B\times \BP C)=[X]$ is a single point, and
$\BP \hat T_{[X]}Seg(\BP B\times \BP C)\supset  \BP T(A^*)$.   Then $T$ is not
$1_A$-generic.
\end{proposition}
\begin{proof}
No element of $\BP \hat T_{[X]}Seg(\BP B\times \BP C)$
has rank greater than two.
\end{proof}

The purpose of stating Proposition \ref{cwflagbb} is to motivate the following theorem:

\begin{theorem}\label{cwflag} 
If  $\BP T(A^*)\cap Seg(\BP B\times \BP C)=[X]$ is a single point, and
$\BP \hat T_{[X]}Seg(\BP B\times \BP C)\cap   \BP T(A^*)$ is a $\pp{m-2}$ and
$T$ is $1_A$-generic, then $T=\tilde T_{m-2,CW}$ is
isomorphic to the Coppersmith-Winograd tensor.
\end{theorem}

\begin{proof}

For the second, we first show that $T$ is $1$-generic.
 If we choose bases such that  $X=b_1\ot c_1$, then, after
changing bases,    the
$\pp{m-2}$ must  be the projectivization of
\be\label{mustbe}
E:=\begin{pmatrix} x_1& x_2 & \cdots & x_{m-1} & 0 \\
x_2 & &   &   & \\
\vdots &  & & & \\
x_{m-1} & & & & \\
 0 & & & &\end{pmatrix} . 
\ene

Write  $T(A^*)=\tspan\{E,M\}$ for some matrix $M$. As $T$ is $1_A$-generic we can assume that $M$ is invertible. In particular, the last row of $M$ must contain
 a  nonzero entry. In the basis order $x_1,\dots,x_{m-1},M$, the space of matrices $T(B^*)$ has triangular form and contains matrices with nonzero diagonal entries. The proof for $T(C^*)$ is analogous, hence $T$ is $1$-generic.

By Proposition \ref{surpriseprop} we may assume that $T(A^*)$ is contained in the space of symmetric matrices. Hence, we may assume that $E$ is as above and $M$ is a symmetric matrix. By further changing the basis we may assume that $M$ has:
\begin{enumerate}
\item the first row and column equal to zero, apart from their last entries that are nonzero (we may assume they are equal to $1$),
\item the last row and column equal to zero apart from their first entries.
\end{enumerate}
Hence the matrix $M$ is determined by a submatrix $M'$ of rows and columns $2$ to $m-1$. As $T(A^*)$ contains a matrix of maximal rank, the matrix $M'$ must have rank $m-2$. We can change the basis $x_2,\dots,x_{m-1}$ in such a way that the quadric corresponding to $M'$ equals $x_2^2+\dots+x_{m-1}^2$. This will also change the other matrices, which correspond to quadrics $x_1x_i$ for $1\leq i\leq m-1$, but will not change the space that they span. We obtain the tensor $\tilde T_{m-2,CW}$, that indeed satisfies the assumptions of the theorem.
 \end{proof}
\subsection{A second geometric characterization of the Coppersmith-Winograd tensors}
 {\it Compression genericity} is defined and discussed in \cite{2016arXiv160807486L}.
Here we just discuss the simplest case.
We say a   $1$-generic, tensor $T\in A\ot B\ot C$ is {\it maximally  
compressible}
if there exists hyperplanes $H_A\subset A^*$, $H_B\subset B^*$, $H_C\subset C^*$
such that
$T\mid_{H_A\times H_B\times H_C}=0$.

If $T\in S^3A\subset A\ot A\ot A$, we will say $T$ is  {\it maximally symmetric compressible} if there exists
a hyperplane $H_A\subset A^*$ such that  $T\mid_{H_A\times H_A\times H_A}=0$.

Recall that a tensor $T\in \BC^m\ot \BC^m\ot \BC^m$ that is $1$-generic and satisfies
Strassen's equations is strictly isomorphic to a tensor in $S^3\BC^m$. 
\begin{theorem}\label{cwcompr}
Let $T\in S^3\BC^m$ be $1$-generic and maximally symmetric compressible.
Then $T$ is one of:
\begin{enumerate}
\item $T_{m-1,CW}$
\item $\tilde T_{m-2,CW}$
\item $T=a_1(a_1^2+\cdots a_m^2)$. As a subspace of $\BC^m\ot \BC^m$,  this is
$$
\begin{pmatrix} x_1 & x_2 & &\cdots & x_m\\
x_2 & x_1 & 0&\cdots & \\
x_3 & 0  & x_1&  & \\
\vdots & 0  &  & \ddots & \\
x_m & 0  &  &  & x_1
\end{pmatrix}.
$$
\end{enumerate}
In particular, the only $1$-generic, maximally symmetric compressible,
minimal border rank tensor in $\BC^m\ot \BC^m\ot \BC^m$ is $\tilde T_{m-2,CW}$.
\end{theorem}
\begin{proof}
Let $a_1$ be a basis of the line  $H_A\upperp\subset \BC^m$.
Then $T=a_1Q$ for some $Q\in S^2\BC^m$. By $1$-genericity, the rank of
$Q$ is either $m$ or $m-1$. If the rank is $m$, there are two
cases, either the hyperplane $H_A$ is tangent to $Q$, or it intersects it
transversely. The second is case 3. The first  has a normal form
$a_1(a_1a_m+ a_2^2+\cdots + a_{m-1}^2)$, which, when written as a tensor,  is $\tilde T_{m-2,CW}$.
If $Q$ has rank $m-1$, by $1$-genericity, its vertex must be  in $H_A$ and thus
we may choose coordinates such that
$Q=(a_2^2+\cdots + a_m^2)$, but then $T$, written as a tensor is $T_{m-1,CW}$.
\end{proof}

\bibliographystyle{amsplain}

\bibliography{Lmatrix}

\def\cdprime{$''$} \def\cprime{$'$} \def\cprime{$'$} \def\cprime{$'$}
  \def\Dbar{\leavevmode\lower.6ex\hbox to 0pt{\hskip-.23ex \accent"16\hss}D}
  \def\cprime{$'$} \def\cprime{$'$} \def\cdprime{$''$} \def\cprime{$'$}
  \def\cprime{$'$} \def\Dbar{\leavevmode\lower.6ex\hbox to 0pt{\hskip-.23ex
  \accent"16\hss}D} \def\cprime{$'$} \def\cprime{$'$} \def\cprime{$'$}
  \def\cprime{$'$} \def\Dbar{\leavevmode\lower.6ex\hbox to 0pt{\hskip-.23ex
  \accent"16\hss}D} \def\cprime{$'$} \def\cprime{$'$}
\providecommand{\bysame}{\leavevmode\hbox to3em{\hrulefill}\thinspace}
\providecommand{\MR}{\relax\ifhmode\unskip\space\fi MR }
\providecommand{\MRhref}[2]{%
  \href{http://www.ams.org/mathscinet-getitem?mr=#1}{#2}
}
\providecommand{\href}[2]{#2}
\begin{thebibliography}{10}

\bibitem{MR3025382}
Boris Alexeev, Michael~A. Forbes, and Jacob Tsimerman, \emph{Tensor rank: some
  lower and upper bounds}, 26th {A}nnual {IEEE} {C}onference on {C}omputational
  {C}omplexity, IEEE Computer Soc., Los Alamitos, CA, 2011, pp.~283--291.
  \MR{3025382}

\bibitem{DBLP:journals/corr/AmbainisFG14}
Andris Ambainis, Yuval Filmus, and Fran{\c{c}}ois~Le Gall, \emph{Fast matrix
  multiplication: Limitations of the laser method}, CoRR \textbf{abs/1411.5414}
  (2014).

\bibitem{MR1579064}
S.~Aronhold, \emph{Theorie der homogenen {F}unctionen dritten {G}rades von drei
  {V}er\"anderlichen}, J. Reine Angew. Math. \textbf{55} (1858), 97--191.
  \MR{1579064}

\bibitem{MR3069955}
Edoardo Ballico and Alessandra Bernardi, \emph{Stratification of the fourth
  secant variety of {V}eronese varieties via the symmetric rank}, Adv. Pure
  Appl. Math. \textbf{4} (2013), no.~2, 215--250. \MR{3069955}

\bibitem{MR2836258}
Daniel~J. Bates and Luke Oeding, \emph{Toward a salmon conjecture}, Exp. Math.
  \textbf{20} (2011), no.~3, 358--370. \MR{2836258 (2012i:14056)}

\bibitem{blaser2014explicit}
Markus Bl{\"a}ser, \emph{Explicit tensors}, Perspectives in Computational
  Complexity, Springer, 2014, pp.~117--130.

\bibitem{MR3368091}
Grigoriy Blekherman and Zach Teitler, \emph{On maximum, typical and generic
  ranks}, Math. Ann. \textbf{362} (2015), no.~3-4, 1021--1031. \MR{3368091}

\bibitem{MR1617757}
Nader~H. Bshouty, \emph{On the direct sum conjecture in the straight line
  model}, J. Complexity \textbf{14} (1998), no.~1, 49--62. \MR{1617757
  (99c:13056)}

\bibitem{MR3092255}
Jaroslaw Buczy{\'n}ski, Adam Ginensky, and J.~M. Landsberg, \emph{Determinantal
  equations for secant varieties and the {E}isenbud-{K}oh-{S}tillman
  conjecture}, J. Lond. Math. Soc. (2) \textbf{88} (2013), no.~1, 1--24.
  \MR{3092255}

\bibitem{MR2996361}
Jaroslaw Buczy{\'n}ski and J.~M. Landsberg, \emph{Ranks of tensors and a
  generalization of secant varieties}, Linear Algebra Appl. \textbf{438}
  (2013), no.~2, 668--689. \MR{2996361}

\bibitem{MR3239293}
\bysame, \emph{On the third secant variety}, J. Algebraic Combin. \textbf{40}
  (2014), no.~2, 475--502. \MR{3239293}

\bibitem{PostBucz}
Jaroslaw Buczynski and Elisa Postinghel, \emph{On strassen's conjecture
  (lecture)}, https://simons.berkeley.edu/talks/elisa-postinghel-2014-11-12
  (2014).

\bibitem{BCS}
Peter B{\"u}rgisser, Michael Clausen, and M.~Amin Shokrollahi, \emph{Algebraic
  complexity theory}, Grundlehren der Mathematischen Wissenschaften
  [Fundamental Principles of Mathematical Sciences], vol. 315, Springer-Verlag,
  Berlin, 1997, With the collaboration of Thomas Lickteig. \MR{99c:68002}

\bibitem{MR3320211}
Enrico Carlini, Maria~Virginia Catalisano, and Luca Chiantini, \emph{Progress
  on the symmetric {S}trassen conjecture}, J. Pure Appl. Algebra \textbf{219}
  (2015), no.~8, 3149--3157. \MR{3320211}

\bibitem{Como94:SP}
P.~Comon, \emph{Independent {C}omponent {A}nalysis, a new concept~?}, Signal
  Processing, Elsevier \textbf{36} (1994), no.~3, 287--314, Special issue on
  Higher-Order Statistics.

\bibitem{copwin135}
Don Coppersmith and Shmuel Winograd, \emph{Matrix multiplication via arithmetic
  progressions}, J. Symbolic Comput. \textbf{9} (1990), no.~3, 251--280.
  \MR{91i:68058}

\bibitem{MR766508}
Ephraim Feig and Shmuel Winograd, \emph{On the direct sum conjecture}, Linear
  Algebra Appl. \textbf{63} (1984), 193--219. \MR{766508 (86h:15022)}

\bibitem{MR1573008}
Ismor Fischer, \emph{Sums of {L}ike {P}owers of {M}ultivariate {L}inear
  {F}orms}, Math. Mag. \textbf{67} (1994), no.~1, 59--61. \MR{1573008}

\bibitem{MR2996364}
Shmuel Friedland, \emph{On tensors of border rank {$l$} in {$\Bbb{C}^{m\times
  n\times l}$}}, Linear Algebra Appl. \textbf{438} (2013), no.~2, 713--737.
  \MR{2996364}

\bibitem{FriedlandGross}
Shmuel Friedland and Elizabeth Gross, \emph{A proof of the set-theoretic
  version of the salmon conjecture}, J. Algebra \textbf{356} (2012), 374--379.
  \MR{2891138}

\bibitem{MR0132079}
Murray Gerstenhaber, \emph{On dominance and varieties of commuting matrices},
  Ann. of Math. (2) \textbf{73} (1961), 324--348. \MR{0132079 (24 \#A1926)}

\bibitem{MR1199042}
Robert~M. Guralnick, \emph{A note on commuting pairs of matrices}, Linear and
  Multilinear Algebra \textbf{31} (1992), no.~1-4, 71--75. \MR{1199042
  (94c:15021)}

\bibitem{MR3236394}
Wolfgang Hackbusch, \emph{Tensor spaces and numerical tensor calculus},
  Springer Series in Computational Mathematics, vol.~42, Springer, Heidelberg,
  2012. \MR{3236394}

\bibitem{MR2202260}
Atanas Iliev and Laurent Manivel, \emph{Varieties of reductions for
  {${\frak{gl}}_n$}}, Projective varieties with unexpected properties, Walter
  de Gruyter GmbH \& Co. KG, Berlin, 2005, pp.~287--316. \MR{MR2202260
  (2006j:14056)}

\bibitem{MR861366}
Joseph Ja'Ja' and Jean Takche, \emph{On the validity of the direct sum
  conjecture}, SIAM J. Comput. \textbf{15} (1986), no.~4, 1004--1020.
  \MR{MR861366 (88b:68084)}

\bibitem{MR813045}
Thomas~J. Laffey, \emph{The minimal dimension of maximal commutative
  subalgebras of full matrix algebras}, Linear Algebra Appl. \textbf{71}
  (1985), 199--212. \MR{813045 (87a:15025)}

\bibitem{MR2865915}
J.~M. Landsberg, \emph{Tensors: geometry and applications}, Graduate Studies in
  Mathematics, vol. 128, American Mathematical Society, Providence, RI, 2012.
  \MR{2865915}

\bibitem{MR3162411}
\bysame, \emph{New lower bounds for the rank of matrix multiplication}, SIAM J.
  Comput. \textbf{43} (2014), no.~1, 144--149. \MR{3162411}

\bibitem{LMsecb}
J.~M. Landsberg and Laurent Manivel, \emph{Generalizations of {S}trassen's
  equations for secant varieties of {S}egre varieties}, Comm. Algebra
  \textbf{36} (2008), no.~2, 405--422. \MR{MR2387532}

\bibitem{2016arXiv160807486L}
J.~M. {Landsberg} and M.~{Michalek}, \emph{{A $2n^2-log(n)-1$ lower bound for
  the border rank of matrix multiplication}}, ArXiv e-prints (2016).

\bibitem{MR2628829}
J.~M. Landsberg and Zach Teitler, \emph{On the ranks and border ranks of
  symmetric tensors}, Found. Comput. Math. \textbf{10} (2010), no.~3, 339--366.
  \MR{2628829 (2011d:14095)}

\bibitem{legall}
Francois Le~Gall, \emph{Powers of tensors and fast matrix multiplication},
  arXiv:1401.7714.

\bibitem{MR3503693}
Arielle Leitner, \emph{Limits under conjugacy of the diagonal subgroup in
  {$SL_n(\Bbb{R})$}}, Proc. Amer. Math. Soc. \textbf{144} (2016), no.~8,
  3243--3254. \MR{3503693}

\bibitem{MR87f:15017}
Thomas Lickteig, \emph{Typical tensorial rank}, Linear Algebra Appl.
  \textbf{69} (1985), 95--120. \MR{87f:15017}

\bibitem{ottrento}
Giorgio Ottaviani, \emph{Symplectic bundles on the plane, secant varieties and
  {L}\"uroth quartics revisited}, Vector Bundles and Low Codimensional
  Subvarieties: State of the Art and Recent Developments (R.~Notari G.~Casnati,
  F.~Catanese, ed.), Quaderni di Matematica, vol.~21, Dip. di Mat., II Univ.
  Napoli, 2007, pp.~315--352.

\bibitem{MR0207178}
V.~Ja. Pan, \emph{On means of calculating values of polynomials}, Uspehi Mat.
  Nauk \textbf{21} (1966), no.~1 (127), 103--134. \MR{0207178}

\bibitem{MR3144910}
Ran Raz, \emph{Tensor-rank and lower bounds for arithmetic formulas}, J. ACM
  \textbf{60} (2013), no.~6, Art. 40, 15. \MR{3144910}

\bibitem{MR623057}
A.~Sch{\"o}nhage, \emph{Partial and total matrix multiplication}, SIAM J.
  Comput. \textbf{10} (1981), no.~3, 434--455. \MR{MR623057 (82h:68070)}

\bibitem{stothers}
A.~Stothers, \emph{On the complexity of matrix multiplication}, PhD thesis,
  University of Edinburgh, 2010.

\bibitem{Strassen505}
V.~Strassen, \emph{Rank and optimal computation of generic tensors}, Linear
  Algebra Appl. \textbf{52/53} (1983), 645--685. \MR{85b:15039}

\bibitem{MR882307}
\bysame, \emph{Relative bilinear complexity and matrix multiplication}, J.
  Reine Angew. Math. \textbf{375/376} (1987), 406--443. \MR{MR882307
  (88h:11026)}

\bibitem{MR0395328}
Volker Strassen, \emph{Evaluation of rational functions}, Complexity of
  computer computations ({P}roc. {S}ympos., {IBM} {T}homas {J}. {W}atson {R}es.
  {C}enter, {Y}orktown {H}eights, {N}.{Y}., 1972), Plenum, New York, 1972,
  pp.~1--10, 187--212. \MR{0395328}

\bibitem{MR0521168}
\bysame, \emph{Vermeidung von {D}ivisionen}, J. Reine Angew. Math. \textbf{264}
  (1973), 184--202. \MR{MR0521168 (58 \#25128)}

\bibitem{MR2118458}
D.~A. Suprunenko and R.~I. Tyshkevich, \emph{Perestanovochnye matritsy}, second
  ed., \`Editorial URSS, Moscow, English translation of first edition: Academic
  Press: New York, 1968, 2003. \MR{2118458 (2006b:16045)}

\bibitem{MR1509924}
Emil Toeplitz, \emph{Ueber ein {F}l\"achennetz zweiter {O}rdnung}, Math. Ann.
  \textbf{11} (1877), no.~3, 434--463. \MR{1509924}

\bibitem{williams}
Virginia Williams, \emph{Breaking the coppersimith-winograd barrier}, preprint.

\bibitem{MR0375839}
Shmuel Winograd, \emph{Some remarks on fast multiplication of polynomials},
  Complexity of sequential and parallel numerical algorithms ({P}roc.
  {S}ympos., {C}arnegie-{M}ellon {U}niv., {P}ittsburgh, {P}a., 1973), Academic
  Press, New York, 1973, pp.~181--196. \MR{0375839}

\bibitem{2015arXiv150405597Z}
J.~{Zuiddam}, \emph{{A note on the gap between rank and border rank}}, ArXiv
  e-prints (2015).

\end{thebibliography}

\end{document}